\documentclass[11pt]{llncs}
\usepackage{latexsym}
\usepackage{times}
\usepackage{amsmath,amsfonts}
\usepackage{fullpage}

\usepackage[dvips]{epsfig}
\graphicspath{{./Fig_eps/}{./Eps/}}
\DeclareGraphicsExtensions{.ps,.eps}
\usepackage{color}

\newcommand{\dima}[1]{{\color{red}{#1}}}

\renewcommand{\dima}[1]{{\color{black}{#1}}}

\title{Algorithmic Barriers from Phase Transitions}

\author{
Dimitris Achlioptas\inst{1}\thanks{Supported by NSF CAREER award CCF-0546900 and an Alfred P. Sloan Fellowship.}
\and Amin Coja-Oghlan\inst{2}\thanks{Supported by the Deutsche Forschungsgemeinschaft (DFG CO 646)}
}
\date{\today}
\institute{
UC Santa Cruz, Santa Cruz, CA 95064, USA,
\email{optas@cs.ucsc.edu}
\and
University of Edinburgh, UK
\email{acoghlan@inf.ed.ac.uk}
}

\newcommand\eps{\varepsilon}

\newcommand\Erw{\mathrm{E}}
\newcommand\ex{\mathbf{E}}
\newcommand\pr{\mathrm{Pr}}

\newcommand\EE{\mathcal{E}}

\newcommand\IS{\Lambda}

\newcommand{\bink}[2]
    {{{#1}\choose {#2}}}

\renewcommand{\bink}[2]
    {\binom{#1}{#2}}

\newcommand\NNN{\mathcal{N}}

\newcommand\SSS{\mathcal{S}}

\newcommand\bc[1]{\left({#1}\right)}
\newcommand\cbc[1]{\left\{{#1}\right\}}
\newcommand\bcfr[2]{\bc{\frac{#1}{#2}}}

\newcommand\brk[1]{\left\lbrack{#1}\right\rbrack}

\newcommand\abs[1]{\left|{#1}\right|}

\newcommand{\Whp}{W.h.p.}
\newcommand{\whp}{w.h.p.}
\newcommand{\wupp}{w.u.p.p.}

\newcommand{\dist}{\mbox{dist}}

\newcommand{\gnp}{G(n,p)}
\newcommand{\gnm}{G_{n,m}}

\newcommand{\clusteroid}{region}
\newcommand{\clusteroids}{regions}

\begin{document}

\spnewtheorem{Algo}[theorem]{Algorithm}{\bfseries}{}

\maketitle

\begin{abstract}

For many random Constraint Satisfaction Problems, by now, we have asymptotically tight estimates of the largest constraint density for which they have solutions. At the same time, all known polynomial-time algorithms for many of these problems already completely fail to find solutions at much smaller densities. For example, it is well-known that it is easy to color a random graph using twice as many colors as its chromatic number. Indeed, some of the simplest possible coloring algorithms already achieve this goal. Given the simplicity of those algorithms, one would expect there is a lot of room for improvement. Yet, to date, no algorithm is known that uses $(2-\epsilon) \chi$ colors, in spite of efforts by numerous researchers over the years. 

In view of the remarkable resilience of this factor of 2 against every algorithm hurled at it, we believe it is natural to inquire into its origin. We do so by analyzing the evolution of the set of $k$-colorings of a random graph, viewed as a subset of $\{1,\ldots,k\}^{n}$, as edges are added. We prove that the factor of 2 corresponds in a precise mathematical sense to a phase transition in the geometry of this set. Roughly, the set of $k$-colorings looks like a giant ball for $k \ge 2 \chi$, but like an error-correcting code for $k \le (2-\epsilon) \chi$. We prove that a completely analogous phase transition also occurs both in random $k$-SAT and in random hypergraph 2-coloring. And that for each problem, its location corresponds precisely with the point were all known polynomial-time algorithms fail. To prove our results we develop a general technique that allows us to prove rigorously much of the celebrated 1-step Replica-Symmetry-Breaking hypothesis of statistical physics for random CSPs.

\end{abstract}

\newpage

\section{Introduction}

For many random Constraint Satisfaction Problems (CSP), such as random graph coloring, random $k$-SAT, random Max $k$-SAT, and hypergraph 2-coloring, by now, we have asymptotically tight estimates for the largest constraint density for which typical instances have solutions (see~\cite{ANP}). At the same time, all known efficient algorithms for each problem fair very poorly, i.e., they stop finding solutions at constraint densities \emph{much} lower than those for  which we can prove that solutions exist. Adding insult to injury, the best known algorithm for each problem asymptotically fairs no better than certain extremely naive algorithms for the problem.

For example, it has been known for nearly twenty years~\cite{ChFrGUC} that the following very simple algorithm will find a satisfying assignment of a random $k$-CNF formula with $m=rn$ clauses for $r = O(2^k/k)$: if there is a unit clause satisfy it; otherwise assign a random value to a random unassigned variable. While it is known that random $k$-CNF remain satisfiable for $r = \Theta(2^{k})$, no polynomial-time algorithm is known to find satisfying assignments for $r = (2^k/k)\cdot \omega(k)$ for some function $\omega(k) \to \infty$. 

Similarly, for all $k\ge 3$, the following algorithm~\cite{mcgrim,focs} will $k$-color a random graph with average degree $d \le k \ln k$: select a random vertex with fewest available colors left and assign it a random available color. While it is known that random graphs remains $k$-colorable for $d \sim 2 \, k \ln k$, no polynomial-time algorithm is known that can $k$-color a random graph of average degree $(1+\epsilon) k \ln k$ for some fixed $\epsilon >0$ and arbitrarily large $k$.  Equivalently, while it is trivial to color a random graph using twice as many colors as its chromatic number, no polynomial-time algorithm is known that can get by with $(2-\epsilon) \chi$ colors, for some fixed $\epsilon > 0$.\medskip

Random $k$-SAT and random graph coloring are not alone. In fact, for nearly every random CSP of interest, the known results establish a completely analogous state of the art:
\begin{enumerate}
\item\label{U}
There is a trivial upper bound on the largest constraint density for which solutions exist. 
\item\label{L}
There is a non-constructive proof, usually via the second moment method, that the bound from~(\ref{U})  is essentially tight, i.e., that solutions do exist for densities nearly as high as the trivial upper bound.
\item\label{S}
Some simple algorithm finds solutions up to a constraint density much below the one from (\ref{L}).
\item\label{A}
No polynomial-time algorithm is known to succeed for a density asymptotically greater than that in (\ref{S}).
\end{enumerate}

In this paper we prove that this is not a coincidence. Namely, for random graph coloring, random $k$-SAT, and random hypergraph 2-coloring, we prove that the point where all known algorithms stop is precisely the point where the geometry of the space of solutions undergoes a dramatic change. This is known as a ``dynamical'' phase transition in statistical physics and our results establish rigorously for random CSPs a large part of the ``1-step Replica Symmetry Breaking'' hypothesis~\cite{pnas}. Roughly speaking, this hypothesis asserts that while the set of solutions for low densities looks like a giant ball, at some critical point this ball shatters into exponentially many pieces that are far apart from one another and separated by huge ``energy barriers''. Algorithms (even extremely simple ones) have no problem finding solutions in the ``ball'' regime, but no algorithm is known that can find solutions in the ``error-correcting code'' regime.

We believe that the presence of dynamical phase transitions in random CSPs is a very general phenomenon, whose qualitative characteristics should be problem-independent, i.e., {\em universal}. The fact that we can establish the exact same qualitative picture for a problem with binary constraints over $k$-ary variables (random graph $k$-coloring) and a problem with $k$-ary constraints over binary variables (hypergraph 2-colorability) certainly lends support to this notion. That said, we wish to emphasize that determining for each random CSP the location of its dynamical phase transition (as we do in this paper for the three problems mentioned, in order to show that the transition coincides with the demise of all known algorithms)  requires non-trivial, problem-specific ideas and computations.

Perhaps the following is an intuitive model of how a dynamical phase transition comes about. In random graph coloring, rather than thinking of the number of available colors as fixed and the constraint density (number of edges) as increasing, imagine that we keep the constraint density fixed, but we keep decreasing the number of available colors. If we start with $q$ available colors where $q \gg \chi$, it is reasonable to imagine that the set of valid $q$-colorings, viewed as a subset of $\{1,2,\ldots,q\}^{n}$, has a nice ``round'' shape, the rounder the greater $q$ is relative to $\chi$. By the same token, when we restrict our attention to the set of those $q$-colorings that only use colors $\{1,2,\ldots,q-1\}$, we are taking a ``slice' of the set of $q$-colorings. With each slicing the connectivity of the set at hands deteriorates, until at some point the set shatters. For example, slicing the 2-dimensional unit sphere through the origin yields a circle, but slicing the circle, yields a pair of points.

We conclude the introduction with a few words about the technical foundation for our work. To prove the existence (and determine the location) of a dynamical phase transition one needs access to statistical properties of the uniform measure over solutions. A geometric way of thinking about this is as follows. Given a CSP instance, say a $k$-CNF formula with $m$ clauses chosen uniformly at random, consider the function $H$ on $\{0,1\}^{n}$ that assigns to each truth assignment the number of clauses it violates. In this manner, $H$ defines a ``landscape'' in which satisfying assignments correspond to valleys at sea-level. Understanding statistical properties of the uniform measure over solutions amounts to understanding ``the view'' one enjoys from such a valley, a probabilistically  formidable task. As we discuss in Section~\ref{departure}, we can establish the following: the number of solutions of a random CSP is sufficiently concentrated around its exponentially large expectation for the view from a random sea-level valley to be  ``the same'' as the view from an ``artificial'' valley.  That is, from the valley that results by first selecting a random $\sigma \in \{0,1\}^{n}$ and then forming a random \mbox{$k$-CNF} formula, also with $m$ clauses, but now chosen uniformly among the clauses satisfied by $\sigma$, i.e., the view from the \emph{planted} satisfying assignment in the planted model. This is a \emph{much} easier view to understand and we believe that the ``transfer'' theorems we establish in this paper will significantly aid in the analysis of random CSPs.

\section{Statement of Results}

To present our results in a uniform manner we need to introduce some common notions. Let $V$ be a set of $n$ variables, all with the same domain $D$, and let $C$ be an arbitrary set of constraints over the variables in $V$. A CSP instance is a subset of $C$. We let $\dist(\sigma,\tau)$ denote the Hamming distance between $\sigma,\tau \in D^{n}$ and we turn $D^n$ into a graph by saying that $\sigma,\tau$ are adjacent if $\dist(\sigma,\tau)=1$. For a given instance $I$, we let $H=H_{I} : D^n \to \mathbb{N}$ be the function counting the number of constraints of $I$ violated by each $\sigma \in D^{n}$. 

\begin{definition}
The \textbf{height} of a path $\sigma_0,\sigma_1,\ldots,\sigma_t \in D^n$ is $\max_i H(\sigma_i)$. We say that $\sigma \in D^{n}$ is a solution of an instance $I$, if $H(\sigma)=0$. We will denote by $\SSS(I)$ the set of all solutions of an instance $I$. The \textbf{clusters} of an instance $I$ are the connected components of $\SSS(I)$. A \textbf{\clusteroid} is a non-empty union of clusters. 
\end{definition}

\begin{remark}
The term cluster comes from physics. Requiring $\dist(\sigma,\tau)=1$ to say that $\sigma,\tau$ are adjacent is somewhat arbitrary (but conceptually simplest) and a number of our results hold if one replaces 1 with $o(n)$.
\end{remark}

%Given a graph $G$ to be $k$-colored, let $H=H_{G} : \{1,\ldots,k\}^n \to \mathbb{N}$ count the number of monochromatic edges in each $k$-partition $\sigma$ of the vertices of $G$. We turn $\{1,\ldots,k\}^n$ into a graph by saying that $\sigma,\tau$ are adjacent if they have Hamming distance 1. The \emph{height} of a path $\sigma_0,\sigma_1,\ldots,\sigma_t \in \{1,\ldots,k\}^n$ is $\max_i H(\sigma_i)$. We let $\SSS_{k}(G) = \{\sigma \in  \{1,\ldots,k\}^n : H_{G}(\sigma) =0 \}$, i.e., $\SSS_{k}(G)$ is the set of valid $k$-colorings $G$. 

We will be interested in distributions of CSP instances as the number of variables $n$ grows. The set $C=C_n$ will typically consist of all possible constraints of a certain type, e.g., the set of all $\binom{n}{k}$ possible hyperedges in the problem of 2-coloring random $k$-uniform hypegraphs. We let $I_{n,m}$ denote the set of all CSP instances with precisely $m$ distinct constraints from $C_n$ and we let $\mathcal{I}_{n,m}$ denote the uniform distribution on the set of all instances $I_{n,m}$. We will say that a sequence of events $\mathcal{E}_{n}$ holds {\em with high probability\/} (w.h.p.) if $\lim_{n\to\infty} \Pr[\mathcal{E}_{n}] = 1$ and {\em with uniformly positive probability\/} (w.u.p.p.) if  $\liminf_{n\to\infty} \Pr[\mathcal{E}_{n}] > 0$. As per standard practice in the study of random structures, we will take the liberty of writing $\mathcal{I}_{n,m}$ to denote the underlying random variable and, thus, write things like ``The probability that  $\SSS(\mathcal{I}_{n,m})$..."

\subsection{Shattering}

\begin{definition}\label{Def_icy}
We say that the set of solutions of $\mathcal{I}_{n,m}$ \textbf{shatters} if there exist constants $\beta,\gamma,\zeta,\theta>0$ such that \whp\ $\SSS(\mathcal{I}_{n,m})$ can be partitioned into \clusteroids\ so that:
\begin{enumerate}
\item
The number of \clusteroids\ is at least $e^{\beta n}$.
\item
Each \clusteroid\ contains at most an $e^{-\gamma n}$ fraction of all solutions.
\item
The Hamming distance between any two \clusteroids\ is at least $\zeta n$.
\item
Every path between vertices in distinct \clusteroids\ has height at least $\theta n$.
\end{enumerate}
\end{definition}

Our first main result asserts that the space of solutions for random graph coloring, random $k$-SAT, and random hypergraph 2-colorability shatters and that this shattering occurs just above the largest density for which any polynomial-time algorithm is known to find solutions for the corresponding problem. Moreover, we prove that the space remains shattered until, essentially, the CSP's satisfiability threshold. More precisely:
\medskip

\noindent -- A random graph with average degree $d$, i.e., $m=dn/2$, is \whp\ $k$-colorable for $d  \le (2-\gamma_{k}) k \ln k$,  where $\gamma_{k} \to 0$. The best poly-time $k$-coloring algorithm \whp\ fails for $d \ge (1+\delta_{k}) k \ln k$, where $\delta_{k} \to 0$. 
\begin{theorem}\label{Thm_icy_col}
There exists a sequence $\epsilon_{k} \to 0$, such that the space of $k$-colorings of a random graph with average degree $d$ shatters for all
\begin{equation}\label{col_range}
(1+\epsilon_{k}) k\ln k \le d \le (2-\epsilon_{k}) k \ln k \enspace .
\end{equation}
\end{theorem}

\noindent -- A random $k$-CNF formula with $n$ variables and $rn$ clauses is \whp\ satisfiable for $r \le 2^{k} \ln 2 -k$.  
The best poly-time satisfiability algorithm \whp\ fails for $r > 2^{k+1}/k$. In~\cite{MontanariGibbs}, non-rigorous, but mathematically sophisticated evidence is given that a different algorithm succeeds for $r=\Theta((2^{k}/k) \ln k)$, but not higher.

\begin{theorem}\label{Thm_icy_sat}
There exists a sequence $\epsilon_{k} \to 0$ such that the space of satisfying assignments of a random $k$-CNF formula with $rn$ clauses shatters for all
\begin{equation}\label{sat_range}
(1+\epsilon_{k}) \, \frac{2^{k}}{k} \, \ln k \le r \le (1-\epsilon_{k}) 2^{k} \ln 2 \enspace .
\end{equation} 
\end{theorem}

\noindent -- A random $k$-uniform hypergraph with $n$ variables and $rn$ edges is \whp\ 2-colorable for $r \le 2^{k-1} \ln 2 - \frac{3}{2}$.  The best poly-time 2-coloring algorithm \whp\ fails for $r > 2^{k}/k$. In~\cite{MontanariGibbs}, non-rigorous, but mathematically sophisticated evidence is given that a different algorithm succeeds for $r=\Theta((2^{k}/k) \ln k)$, but not higher.

\begin{theorem}\label{Thm_icy_hyper}
There exists a sequence $\epsilon_{k} \to 0$ such that the space of 2-colorings of a random $k$-uniform hypergraph with $rn$ edges shatters for all
\begin{equation}\label{hyper_range}
(1+\epsilon_{k}) \, \frac{2^{k-1}}{k} \, \ln k \le r \le (1-\epsilon_{k}) 2^{k-1} \ln 2 \enspace .
\end{equation} 
\end{theorem}

\begin{remark}
As the notation in~Theorems~\ref{Thm_icy_col},\ref{Thm_icy_sat},\ref{Thm_icy_hyper} is asymptotic in $k$, the stated intervals may be empty for small values of $k$. In this extended abstract we have not optimized the proofs to deliver the smallest values of $k$ for which the intervals are non-empty. Quick calculations suggest $k \ge 6$ for hypergraph 2-colorability, $k\geq 8$ for $k$-SAT, and $k \ge 20$ for $k$-coloring.
\end{remark}

\subsection{Rigidity}

The \clusteroids\ mentioned in~Theorems~\ref{Thm_icy_col}, \ref{Thm_icy_sat} and \ref{Thm_icy_hyper} can be thought of as forming an error-correcting code in the solution-space of each problem. To make this precise we need to introduce the following definition and formalize the notion of  ``a random solution of a random instance''. 
\begin{definition}
Given an instance $I$, a solution $\sigma \in \SSS(I)$ and a variable $v \in V$, we say that $v$ in $(I,\sigma)$:
\begin{itemize}
\item[--]
Is $f(n)$-\textbf{rigid},  if every $\tau \in \SSS(I)$ such that $\tau(v) \neq \sigma(v)$ has $\dist(\sigma,\tau) \ge f(n)$. 
\item[--]
Is $f(n)$-\textbf{loose}, if for every $j \in D$, there exists $\tau \in \SSS(I)$ such that $\tau(v) = j$ and $\dist(\sigma,\tau) \le f(n)$.
\end{itemize}
\end{definition}

We will prove that while before the phase transition, in a typical solution, every variable is loose, after the phase transition nearly every variable is rigid. To formalize the notion of a random/typical solution, recall that $I_{n,m}$ denotes the set of all instances with $m$ constraints over $n$ variables and let $\IS=\IS_{n,m}$ denote the set of all instance--solution pairs, i.e., $\IS_{n,m} = \{(I,\sigma): I\in {I}_{n,m},\,\sigma\in\SSS(I)\}$. We let $\mathcal{U}=\mathcal{U}_{n,m}$ be the probability distribution induced on $\IS_{n,m}$ by the following:
\begin{center}
Choose an instance $I\in I_{n,m}$ uniformly at random.\\ 
If $\SSS(I) \neq \emptyset$, select $\sigma\in\SSS(I)$ uniformly at random. 
\end{center}
We will refer to instance-solution pairs generated according to $\mathcal{U}_{n,m}$ as \textbf{uniform} instance-solution pairs. We note that although the definition of uniform pairs allows for $\SSS(I)$ to be typically empty, i.e., to be in the typically unsatisfiable regime, we will only employ the definition for constraint densities such that \whp\ $\SSS(I)$ contains exponentially many solutions. Hence, our liberty in also using the term a ``typical'' solution.

\begin{theorem}\label{thm:froz}
Let $(I,\sigma)$ be a uniform instance-solution pair where:
\begin{itemize}
\item
$I$ is a graph with $dn/2$ edges, where $d$ is as in~\eqref{col_range}, and $\sigma$ is a $k$-coloring of $I$, or,
\item
$I$ is a $k$-CNF formula with $rn$ clauses, where $r$ is as in~\eqref{sat_range}, and $\sigma$ is a satisfying assignment of $I$, or,
\item
$I$ is a $k$-uniform hypergraph with $rn$ edges, where $r$ is as in~\eqref{hyper_range}, and $\sigma$ is a 2-coloring of $I$. 
\end{itemize}
\Whp\ the number of rigid variables in $(I,\sigma)$ is at least $\gamma_{k}n$, for some sequence $\gamma_{k} \to 1$.\end{theorem}

\begin{remark}
Theorem~\ref{thm:froz} is tight since for every finite constraint density, a random instance \whp\ has $\Omega(n)$ variables that are not bound by any constraint.
\end{remark}

The picture drawn by Theorem~\ref{thm:froz}, whereby nearly all variables are rigid in typical solutions above the dynamical phase transition, is in sharp contrast with our results for densities below the transition for graph coloring and hypergraph 2-colorability. While we believe that an analogous picture holds for $k$-SAT, see Conjecture~\ref{conj:sat}, for technical reasons we cannot establish this presently. (We discuss the additional difficulties imposed by random $k$-SAT in Section~\ref{departure}.)

\begin{theorem}\label{thm:loose}
Let $(I,\sigma)$ be a uniform instance-solution pair where:
\begin{itemize}
\item
$I$ is a graph with $dn/2$ edges, where $d \le (1-\epsilon_k)k \ln k$, and $\sigma$ is a $k$-coloring of $I$, or,
\item
$I$ is a $k$-uniform hypergraph with $rn$ edges, where $r \le (1-\epsilon_k) (2^{k-1}/k) \ln k$, and $\sigma$ is \mbox{a 2-coloring of $I$.} 
\end{itemize}
There exists a sequence $\epsilon_k\rightarrow 0$ such that \whp\ every variable in $(I,\sigma)$ is $o(n)$-loose.
\end{theorem}

We note that in fact, for all $d$ and $r$ as in Theorem~\ref{thm:loose}, \wupp\ $(I,\sigma)$ is such that changing the color of any vertex to any color only requires changing the color of $O(\log n)$ other vertices.

\begin{conjecture}\label{conj:sat}
Let $(I,\sigma)$ be a uniform instance-solution pair where $I$ is a $k$-CNF formula with $rn$ clauses, where $r \le (1-\epsilon_{k})(2^{k}/k) \ln k$, and $\sigma$ is a satisfying assignment of $I$. There exists a sequence $\epsilon_k\rightarrow 0$ such that \whp\ every variable in $(I,\sigma)$ is $o(n)$-loose.
\end{conjecture}

\section{Background and Related Work}

\subsection{Algorithms}

Attempts for a ``quick improvement'' upon either of the naive algorithms mentioned in the introduction for satisfiability/graph coloring, stumble upon the following general fact. Given a CSP instance, consider the bipartite graph in which every variable is adjacent to precisely those constraints in which it appears, known as the factor graph of the instance. For random formulas/graphs, factor graphs are locally tree-like, i.e., for any arbitrarily large constant $D$, the depth-$D$ neighborhood of a random vertex is a tree \whp\ In other words, locally, random CSPs are trivial, e.g., random graphs of any finite average degree are locally 2-colorable. Moreover, as the constraint density is increased, the factor graphs of random CSPs get closer and closer to being biregular, so that degree information is not useful either. Combined, these two facts render all known algorithms impotent, i.e., as the density is increased, their asymptotic performance matches that of trivial algorithms.

In~\cite{MPZ},  M\'{e}zard, Parisi, and Zecchina proposed a new satisfiability algorithm called Survey Propagation (SP) which performs extremely well experimentally on instances of random 3-SAT. This was very surprising at the time and allowed for optimism that, perhaps, random $k$-SAT instances might not be so hard. Moreover, SP was generalized to other problems, e.g., $k$-coloring~\cite{spcol} and Max $k$-SAT~\cite{spmax}. An experimental evaluation of SP for values of $k$ even as small as 5 or 6 is already somewhat problematic, but to the extent it is reliable it strongly suggests that SP does not find solutions for densities as high as those for which solutions are known to exist. Perhaps more importantly, it can be shown that for densities at least as high as $2^k \ln2 - k$,  if SP can succeed at its main task (approximating the marginal probability distribution of the variables with respect to the uniform measure over satisfying assignments),  so can a much simpler algorithm, namely Belief Propagation (BP), i.e., dynamic programming on trees. 

The trouble is that to use either BP or SP to find satisfying assignments one sets variables iteratively. So, even if it is possible to compute approximately correct marginals at the beginning of the execution (for the entire formula), this can stop being the case after some variables are set. Concretely, in~\cite{MontanariGibbs}, Montanari et al.\ showed that (even within the relatively generous assumptions of statistical physics computations) the following Gibbs-sampling algorithm fails above the $(2^k/k) \ln k $ barrier, i.e., step~\ref{bp_comp} below fails to converge after only a small fraction of all variables have been assigned a value:
\begin{enumerate}
\item\label{begin}
Select a variable $v$ at random.
\item\label{bp_comp}
Compute the marginal distribution of $v$ using Belief Propagation.
\item
Set $v$ to $\{0,1\}$ according to the computed marginal distribution; simplify the formula; go to step~\ref{begin}.
\end{enumerate}
 
\subsection{Relating the Uniform and the Planted Model.}

The idea of deterministically embedding a property inside a random structure is very old and, in general, the process of doing this is referred to as ``planting'' the property. In our case, we plant a solution $\sigma$ in a random CSP, by only including constraints compatible with $\sigma$. Juels and Peinado~\cite{Juels} were perhaps the first to explore the relationship between the planted and the uniform model and they did so for the clique problem in dense random graphs $G_{n,1/2}$, i.e., where each edge appears independently with probability 1/2. They showed the distribution resulting from first choosing $G=G_{n,1/2}$ and then planting a clique of size $(1+\eps)\log_2n$ is very close to $G_{n,1/2}$ and suggested this as a scheme to obtain a one-way-function. Since the planted clique has size only $(1+\eps)\log_2n$, the basic argument in~\cite{Juels} is closely related to subgraph counting. In contrast, the objects under consideration in our work ($k$-colorings, satisfying assignments, etc.) have an immediate impact on the \emph{global} structure of the combinatorial object being considered, rather than just being local features, such as a clique on $O(\log n)$ vertices.

Coja-Oghlan, Krivelevich, and Vilenchik~\cite{KV1,KV2} proved that for constraint densities well above the threshold for the existence of solutions, the planted model for $k$-coloring and $k$-SAT is equivalent to the uniform distribution \emph{conditional} on the (exponentially unlikely) existence of at least one solution. In this conditional distribution as well as in the high-density planted model, the geometry of the solution space is very simple, as there is precisely one cluster of solutions.

\subsection{Solution-space Geometry}

In~\cite{fede,phys_clus_j} the first steps were made towards understanding the solution-space geometry of random $k$-CNF formulas by proving the existence of shattering and the presence of rigid variables for $r = \Theta(2^k)$. This was a far cry from the true $r  \sim (2^k/k) \ln k$ threshold for the onset of both phenomena, as we establish here. Besides the quantitative aspect, there is also a fundamentally important difference in the methods employed in~\cite{fede,phys_clus_j} vs.\ those employed here. In those works, properties were established by taking a union bound over all satisfying assignments. It is not hard to show that the derived results are best possible using those methods and, in fact, there is good reason to believe that the results are genuinely tight, i.e., that for densities $o(2^{k})$ the derived properties simply do not hold for \emph{all} satisfying assignments. Here, we instead establish a systematic connection between the planted model and the process of sampling a random solution of a random instance. This argument allows us to analyze ``typical'' solutions while allowing for the possibility that a (relatively small, though exponential) number of ``atypical'' solutions exist. Therefore, we are for the first time in a position to analyze the extremely complex energy landscape of below-threshold instances of random CSPs, and to estimate quantities that appeared completely out of reach prior to this work.

\section{Our Point of Departure: Symmetry, Randomness and Inversion}\label{departure}

As mentioned, the results in this paper are enabled by a set of technical lemmas that allow one to reduce the study of ``random solutions of random CSP instances'' to the study of ``planted CSP solutions''. The conceptual origin of these lemmas can be traced to the following humble  observation.

Let $M$ be an arbitrary $0$-$1$ matrix with the property that all its rows have the same number of 1s and all its columns have the same the number of 1s. A moment's reflection makes it clear that for such a matrix, both of the following methods select a uniformly random 1 from the entire matrix: 
\begin{enumerate}
\item Select a uniformly random column and then a uniformly random 1 in that column.
\item Select a uniformly random row and then a uniformly random 1 in that row.
\end{enumerate}

An example of how we employ this fact for random CSPs is as follows. Let $\mathcal{F}$ be the set of all $k$-CNF formulas with $n$ variables and $m$ distinct clauses (chosen among all $2^{k} \binom{n}{k}$ possible $k$-clauses). Say that $\sigma \in \{0,1\}^n$ NAE-satisfies a formula $F \in \mathcal{F}$ if under $\sigma$, every clause of $F$ has at least one satisfied and at least one falsified literal. Let $M$ be the $2^{n} \times |\mathcal{F}|$ matrix  where$M_{\sigma,F}=1$ iff $\sigma \in \{0,1\}^{n}$ NAE-satisfies $F$. By the symmetry of $\mathcal{F}$, it is clear that all rows of $M$ have the same number of 1s. Imagine, for a moment, that the same was true for all columns. Then, a uniformly random solution of a uniformly random instance would be distributed {\em exactly\/} as a ``planted'' instance-solution pair: first select $\sigma\in \{0,1\}^{n}$ uniformly at random; then select $m$ distinct clauses uniformly at random among all $2^{k-1} \binom{n}{k}$ clauses NAE-satisfied by $\sigma$.

Our contribution begins with the realization that exact row- and column-balance is not necessary. Rather, it is enough for the 1s in $M$ to be ``well-spread". More precisely, it is enough that the marginal distributions induced on the rows and columns of $M$ by  selecting a uniformly random 1 from the entire matrix are both ``reasonably close to'' uniform. For example, assume we can prove that $\Omega(|\mathcal{F}|)$ columns of $M$ have $\Theta(f(n))$ 1s, where $f(n)$ is the average number of 1s per column.
Indeed, this is precisely the kind of property  implied by the success of the second moment method for random NAE-$k$-SAT~\cite{nae}. Under this assumption, proving that a property holds \wupp\ for a uniformly random solution of a uniformly random instance, reduces to proving that it holds \whp\ for the planted solution of a planted instance, a dramatically simpler task.

There is a geometric intuition behind our transfer theorems which is more conveniently described when every constraint is included independently with the same probability $p$, i.e., we take $p=m/\left(2^{k} \binom{n}{k}\right)$. For all $k\ge 3$ and $m=rn$, it was shown in~\cite{nae} that the resulting instances \wupp\ have exponentially many solutions for $r \le 2^{k-1} \ln 2 - 3/2$. Consider now the following way of generating {\em planted\/} NAE $k$-SAT instances. First, select a formula $F$ by including each clause with probability $p$, exactly as above. Then, select $\sigma \in \{0,1\}^n$ uniformly at random and remove from $F$ all constraints violated by $\sigma$. Call the resulting instance $F'$. Our results say that as long as $q\equiv r(1-2^{-k+1}) \le 2^{k-1} \ln 2 - 3/2$, the instance $F'$ is  ``nearly indistinguishable" from a {\em uniform\/} instance created by including each clause with probability $q$. (We will make this statement precise shortly.)

To see how this happens, recall the function $H: \sigma \to \mathbb{N}$ counting the number of violated constraints under each assignment. Clearly, selecting $F$ specifies such a function $H_F$, while selecting $\sigma \in \{0,1\}^n$ and removing all constraints violated by $\sigma$ amounts to modifying $H_F$ so that $H_F(\sigma) = 0$. One can imagine that such a modification creates a gradient in the vicinity of $\sigma$, a ``crater" with $\sigma$ at its bottom. What we prove is that as long as $H_F$ already had an exponential number of craters and the number of craters  is concentrated, adding one more crater does not make a big difference. Of course, if the density is increased further, the opened crater becomes increasingly obvious, as it takes a larger and larger cone to get from the typical values of $H_F$ down to 0. Hence the ease with which algorithms solve planted instances of high density.\medskip

To prove our transfer theorems we instantiate this idea for random graph $k$-coloring, random $k$-uniform hypergraph $2$-coloring, and random $k$-SAT. For this, a crucial step is deriving a lower bound on the number of solutions of a random instance. For example, in the case of random graph $k$-coloring, we prove that the number of  $k$-colorings, $|\SSS(I_{n,m})|$, for a random graph with $n$ vertices and $m$ edges is ``concentrated'' around its expectation in the sense that \whp
	\begin{equation}\label{eqColLower}
	n^{-1}\; \left|\,\ln|\SSS(I_{n,m})|-\ln\ex\left(|\SSS(I_{n,m})|\right)\,\right|=o(1) \enspace .
	\end{equation}
To prove this, we use the upper bound on the second moment $\ex \brk{|\SSS(I_{n,m})|^2}$ from~\cite{kcol}
to show that \wupp\ $|\SSS(I_{n,m})| = \Omega(\ex |\SSS(I_{n,m})|)$. 
Then, we perform a sharp threshold analysis, using theorems of Friedgut~\cite{EhudHunting},
to prove that~(\ref{eqColLower}) holds, in fact,  with \emph{high} probability.
A similar approach applies to hypegraph $2$-coloring.

The situation for random $k$-SAT is more involved. Indeed, we can prove that the number of satisfying assignments is \emph{not} concentrated around its expectation in the sense of~(\ref{eqColLower}). This problem is mirrored by the fact that the second moment of the number of satisfying assignments exceeds the square of the first moment by an exponential factor (for any constraint density).
Nonetheless, letting $F_{k}(n,m)$ denote a uniformly random $k$-CNF formula with $n$ variables and $m$ clauses, combining techniques from~\cite{yuval} with a sharp threshold analysis, we can derive a lower bound on the number of satisfying assignments that holds \whp, namely
	$n^{-1}\ln|\SSS(F_k(n,m))|\geq n^{-1}\ln\ex|\SSS(F_k(n,m))|-\phi(k)$,
where $\phi(k)\rightarrow0$ exponentially with $k$.
This estimates allows us to approximate the uniform model by the planted model sufficiently
well in order to establish Theorems~\ref{Thm_icy_sat} and~\ref{thm:froz}.

\section{Proof sketches}
 
Due to the space constraints, in the remaining pages we give proof sketches of our results for $k$-coloring, to offer a feel of the transfer theorems and of the style of the arguments one has to employ given those theorems (actual proofs appear in the Appendix). The proofs for hypergraph 2-coloring are relatively similar, as it is also a ``symmetric'' CSP and the second moment methods works on its number of solutions. For $k$-SAT, though, a significant amount of additional work is needed, as properties must be established with exponentially small error probability to overcome the large deviations in the number of satisfying assignments (proofs appear in the Appendix).

\subsection{Transfer Theorem for Random Graph Coloring}

We consider a fixed number $\eps>0$ and assume that $k\geq k_0$ for some sufficiently large $k_0=k_0(\eps)$. We denote $\{1,\ldots,k\}$ as $\brk{k}$.
We are interested in the  probability distribution $ \mathcal{U}_{n,m}$ on $\IS_{n,m}$ resulting from
first choosing a random graph $G=G(n,m)$ and then a random $k$-coloring of $G$ (if one exists).
To analyze this distribution, we consider the distribution $\mathcal{P}_{n,m}$ on $\Lambda_{n,m}$ induced
by following expermient.
\begin{description}
\item[P1.] Generate a uniformly random $k$-partition $\sigma\in \brk{k}^n$.
\item[P2.] Generate a graph $G$ with $m$ edges chosen uniformly at random among the edges bicolored under $\sigma$.
\item[P3.] Output the pair $(G,\sigma)$.
\end{description}
The distribution $\mathcal{P}_{n,m}$ is known as the \emph{planted model}.

\begin{theorem}\label{Thm_ColorExchangeMain}
Suppose that $d=2m/n \le (2-\eps)k\ln k$.
There exists a function $f(n)=o(n)$ such that the following is true.
Let $\mathcal{D}$ be any graph property such that
$G(n,m)$ has $\mathcal{D}$ with probability $1-o(1)$,
and let $\EE$ be any property of pairs $(G,\sigma)\in\Lambda_{n,m}$.
If  for all sufficiently large $n$
	\begin{equation}\label{eqColorExchangeMain}
	\pr_{\mathcal{P}_{n,m}}\brk{(G,\sigma)\mbox{ has }\EE|G\mbox{ has }\mathcal{D}}
		\geq1-\exp(-f(n)),
	\end{equation}
then
	$\pr_{\mathcal{U}_{n,m}}\brk{(G,\sigma)\mbox{ has }\EE}=1-o(1).$
\end{theorem}

\subsection{Loose Variables Below the Transition}

Suppose that $d\leq (1-\eps)k\ln k$. Recall that a graph with vertex set $V$ is said to be $\zeta$-choosable if for any assignments of color lists of length at least $\zeta$ to the elements of $V$, there is a proper coloring in which every vertex receives a color from its list. To prove Theorem~\ref{thm:loose}, we consider the property $\EE$ that all vertices are $o(n)$-loose and the following condition $\mathcal{D}$:
\begin{quote}
For any set $S\subset V$ of size $|S|\leq g(n)$ the subgraph
induced on $S$ is $4$-choosable. 
\end{quote}
Here $g(n)$ is some function such that $f(n)\ll g(n)=o(n)$,
where $f(n)$ is the function from Theorem~\ref{Thm_ColorExchangeMain}.
A standard argument shows that a random graph $G(n,m)$, where $m=O(n)$, satisfies $\mathcal{D}$ \whp
	
By Theorem~\ref{Thm_ColorExchangeMain}, we are thus left to establish~(\ref{eqColorExchangeMain}). Let $\sigma\in\brk{k}^n$ be a uniformly random $k$-partition, and let $G$ be a random graph with $m$ edges such that $\sigma$ is a $k$-coloring of $G$. Since $\sigma$ is uniformly random, we may assume that the color classes $V_i=\sigma^{-1}(i)$
satisfy $|V_i|\sim n/k$. Let $v_0\in V$ be any vertex, and let $l\not=\sigma(v_0)$ be the ``target color'' for $v_0$. Our goal is to find a coloring $\tau$ such that $\tau(v_0)=l$ and $\dist(\sigma,\tau) \le g(n)$.

If $v$ has no neighbor in $V_l$, then we can just assign this color to $v_0$.
Otherwise, we run the following process.
In the course of the process, every vertex is either \emph{awake}, \emph{dead}, or \emph{asleep}.
Initially, all the neighbors of $v_0$ in $V_l$ are awake, $v$ is dead, and all other vertices are asleep.
In each step of the process, pick an awake vertex $w$ arbitrarily and declare it dead
(if there is no awake vertex, terminate the process).
If there are at least five colors $c_1(w),\ldots,c_5(w)$ available such that
$w$ has no neighbor in $V_{c_i(w)}$, then we do nothing.
Otherwise, we pick five colors $c_1(w),\ldots,c_5(w)$ randomly and
declare all asleep neighbors of $w$ in $V_{c_j(w)}$ awake for $1\leq j\leq 5$.

\begin{lemma}\label{Lemma_fewDeadMain}
With probability at least $1-\exp(-f(n))$ there are
at most $g(n)$ dead vertices when the process terminates.
\end{lemma}
The proof of Lemma~\ref{Lemma_fewDeadMain} is based on relating our process to a subcritical branching process.
The basic insight here is that when $d<(1-\eps)k\ln k$ it is very likely
that a vertex $w$ has five immediately available colors.
More precisely, for any $w$ the number of neighbors in any class $V_i$
with $i\not=\sigma(w)$ is asymptotically Poisson with mean
$(1+o(1))\frac{2m}{(k-1)n}\leq(1-\eps+o(1))\frac{k\ln k}{k-1}$.
Hence, the probability that $w$ does \emph{not} have a neighbor in
$V_i$ is about $k^{\eps-1}$. As there are $k$ colors in total, we expect about $(k-1)^{\eps}$
colors available for $w$, i.e., a lot.

To obtain a new coloring $\tau$ in which $v_0$ takes color $l$
we consider the set $D$ of all dead vertices.
We let $\tau(u)=\sigma(u)$ for all $u\in V\setminus D$.
Moreover, conditioning on the event $\mathcal{D}$,
we can assign to each $w\in D$ a color from the list
$\{c_1(w),\ldots,c_5(w)\}\setminus\{l\}$.
Thus, the new coloring $\tau$ differs from $\sigma$ on at most $|D|\leq g(n)=o(n)$ vertices.

\subsection{Rigid Variables Above the Transition}

Suppose that $d\geq (1+\eps)k\ln k$.
To prove Theorem~\ref{thm:froz} for coloring we apply Theorem~\ref{Thm_ColorExchange} as follows. 
We let $\alpha,\beta>0$ be sufficiently small numbers and denote by
$\EE$ the following property of  a pair $(G,\sigma)\in\Lambda_{n,m}$:
	\begin{equation}\label{eqG*Main}
	\parbox{15cm}{
	There is a subgraph $G_*\subset G$ of size $|V(G_*)|\geq(1-\alpha)n$
	such that for every vertex $v$ of $G_*$ and each color $i\not=\sigma(v)$ there are
	at least $\beta\ln k$ vertices vertex $w$ in $G_*$ that are adjacent to $v$
	such that $\sigma(w)=i$.}
	\end{equation}
Also, we let $\mathcal{D}$ be the property that the maximum degree is at most $(\ln n)^{2}$.

We shall prove that for a pair $(G,\sigma)$ chosen from $\mathcal{U}_{n,m}$
a subgraph $G_*$ as in~(\ref{eqG*Main}) exists \whp\
If that is so, then every vertex in $G_*$ has at least one neighbor in every color class other than its own.
Therefore, it is impossible to just assign a different color to any vertex in $G_*$.
In fact, since all vertices in $G_*$ have \emph{a lot} (namely, at least $\beta\ln k$)
of neighbors with every other color, the expansion properties of the random graph $G(n,m)$ imply that
recoloring any vertex $v$ in $G_*$ necessitates the recoloring of at least
$n/(k\ln k)$ further vertices. Loosely speaking, the conflicts resulting from recoloring $v$ spread so rapidly
that we necessarily end up recoloring a huge number of vertices. Thus, all vertices in $G_*$ are $n/(k\ln k)$-rigid. Note that we can not hope for much better, as we can always recolor $v$ by swapping two color classes, i.e., $\sim 2n/k$ vertices.\smallskip

To prove the existence of the subgraph $G_*$, we establish the following.
\begin{lemma}\label{Lemma_coreMain}
Condition~(\ref{eqColorExchangeMain}) holds for $\mathcal{D}$ and $\EE$ as above.
\end{lemma}

To obtain Lemma~\ref{Lemma_coreMain},
let $(G,\sigma)\in\Lambda_{n,m}$ be a random pair chosen from the distribution $\mathcal{P}_{n,m}$.
We may assume that $|\sigma^{-1}(i)|\sim n/k$ for all $i$.
To obtain the graph $G_*$, we perform a ``stripping process''.
As a first step, we obtain a subgraph $H$
by removing from $G$ all vertices that have fewer than $\gamma\ln k$ neighbors in any color
class other than their own.
If $\gamma=\gamma(\eps)$ is sufficiently small, then the expected number
of vertices removed in this way is less than $nk^{-\delta}$ for a $\delta>0$,
because for each vertex $w$ the expected number of neighbors in another color class
is bigger than $(1+\eps)\ln k$.
Then, we keep removing vertices from $H$ that have ``a lot'' of neighbors outside of $H$.
Given the event $\mathcal{D}$, we then show that with probabiltiy $1-\exp(-\Omega(n))$ the final result of this
process is a subgraph $G_*$ that satisfies~(\ref{eqG*Main}).

\subsection{Proof of Theorem~\ref{Thm_icy_col}}

Theorem~\ref{Thm_icy_col} concerns the ``view'' from a random coloring $\sigma$ of $G(n,m)$.
Basically, our goal is to show that only a tiny fraction of all possible colorings 
are ``visible'' from $\sigma$, i.e., $\sigma$ lives in a small, isolated valley.
To establish the theorem, we need a way to measure how ``close'' two colorings $\sigma,\tau$ are.
The Hamming distance is inappropriate here because 
two colorings $\sigma,\tau$ can be at Hamming distance $n$,
although $\tau$ simply results
from permuting the color classes of $\sigma$, i.e.,
although $\sigma$ and $\tau$ are essentially identical.
Instead, we shall use the following concept.
Given two coloring $\sigma,\tau$, we let $M_{\sigma,\tau}=(M^{ij}_{\sigma,\tau})_{1\leq i,j\leq k}$ be the matrix with entries
	$$M^{ij}_{\sigma,\tau}=n^{-1}|\sigma^{-1}(i)\cap\tau^{-1}(j)|.$$
%In words, the $ij$-entry of $M_{\sigma,\tau}$ equals the proportion of vertices that have color $i$ in $\sigma$ and color $j$ in $\tau$.
To measure how close $\tau$ is to $\sigma$ we let
	$$f_{\sigma}(\tau)=\|M_{\sigma,\tau}\|_F^2=\sum_{i,j=1}^k(M^{ij}_{\sigma,\tau})^2 \enspace ,$$
be the squared Frobenius norm of $M_{\sigma,\tau}$. Observe that this quantity reflects the probability that a single random edge is monochromatic under both $\sigma$ and $\tau$, i.e., the correlation of $\sigma$ and $\tau$, precisely as desired. 
Hence, $f_{\sigma}$ is a map from the set $\brk{k}^n$ of $k$-partitions to the interval $\brk{k^{-2},f_{\sigma}(\sigma)}$,
where $f_{\sigma}(\sigma)\geq k^{-1}$.
Thus, the \emph{larger} $f_{\sigma}(\tau)$, the more $\tau$ resembles $\sigma$.
Furthermore, for a fixed $\sigma\in\SSS(G)$ and a number $\lambda>0$ we let
	$$g_{\sigma,G,\lambda}(x)=|\{\tau\in[k]^n:f_\sigma(\tau)=x\wedge H(\tau)\leq\lambda n\}|.$$

In order to show that $\SSS(\gnm)$ with $m=rn$ decomposes into exponentially many regions, we employ the following lemma.

\begin{lemma}\label{Lemma_cluster_colMain}
Suppose that $r>(\frac12+\eps_k)k\ln k$.
There are numbers $k^{-2}<y_1<y_2<k^{-1}$ and $\lambda,\gamma>0$ such that
with high probability 
 a pair $(G,\sigma)\in\Lambda_{n,m}$ chosen from the distributoin $\mathcal{U}_{n,m}$
 has the following two properties.
\begin{enumerate}
\item For all $x\in\brk{y_1,y_2}$ we have $g_{\sigma,G,\lambda}(x)=0$.
\item The number of colorings $\tau\in\SSS(G)$ such that $f_{\sigma}(\tau)>y_2$ is at most $\exp(-\gamma n)\cdot|\SSS(G)|$.
\end{enumerate}
\end{lemma}
Let $G=\gnm$ be a random graph and call $\sigma\in\SSS(G)$ \emph{good} if both (1) and (2) hold.
Then Lemma~\ref{Lemma_cluster_colMain} states that \whp\
a $1-o(1)$-fraction of all $\sigma\in\SSS(G)$ are good.
Hence,  to decompose $\SSS(G)$ into regions, we proceed as follows.
For each $\sigma\in \SSS(G)$ we let
	$\mathcal{C}_\sigma=\{\tau\in\SSS(G):f_\sigma(\tau)>y_2\}.$
\dima{Then starting with the set $S=\SSS(G)$ and removing iteratively some $\mathcal{C}_\sigma$ for a good $\sigma\in S$ yields an exponential number of regions.}
Furthermore, each such region $\mathcal{C}_\sigma$ is separated by a linear Hamming distance
from the set $\SSS(G)\setminus \mathcal{C}_\sigma$, because $f_\sigma$ is ``continuous'' with respect to $n^{-1}\times$Hamming distance.
Thus, Theorem~\ref{Thm_icy_col} follows from Lemma~\ref{Lemma_cluster_colMain}.

Finally, by Theorem~\ref{Thm_ColorExchangeMain}, to prove Lemma~\ref{Lemma_cluster_colMain} it is sufficient to show the following.

\begin{lemma}\label{Lemma_cluster_col_plMain}
Suppose that $r>(\frac12+\eps_k)k\ln k$.
There are $k^{-2}<y_1<y_2<k^{-1}$ and $\lambda,\gamma>0$ such that
with probability at least $1-\exp(-\Omega(n))$
a pair $(G,\sigma)\in\Lambda_{n,m}$ chosen from the distributoin $\mathcal{P}_{n,m}$ has
the two properties stated in Lemma~\ref{Lemma_cluster_colMain}.
\end{lemma}
The proof of Lemma~\ref{Lemma_cluster_col_plMain} is based on the ``first moment method''.
That is, for any $k^{-2}<y<k^{-1}$ we compute the \emph{expected} number
of assignments $\tau\in\brk{k}^n$ such that $f_{\sigma}(\tau)=y$ and $H(\tau)\leq\lambda n$.
This computation is feasible in the planted model and yields
similar expressions as encountered in~\cite{kcol} in the course of computing the second
moment of the number of $k$-colorings.
Therefore, we can show that the expected number of such assignments $\tau$
is exponentially small for a regime $y_1<y<y_2$, whence Lemma~\ref{Lemma_cluster_col_pl} follows from Markov's inequality.

\newpage

\begin{appendix}

In this appendix we present the proofs of our results for graph coloring and random $k$-SAT, thereby presenting the most important techniques. We omit the proofs for hypergraph 2-coloring, as these are similar to but simpler
than the $k$-SAT proofs, due to the fact that the transfer theorem for hypergraph 2-coloring is as strong as that for $k$-coloring. We generally assume that $n$ is sufficiently large.

\section{Graph coloring}

\subsection{The planted model}

In this section we consider a fixed number $\eps>0$ and assume that $k\geq k_0$
for some sufficiently large $k_0=k_0(\eps)$.
We are interested in the  probability distribution $ \mathcal{U}_{n,m}$ on $\IS_{n,m}$.
To analyze this distribution, we consider the distribution $\mathcal{P}_{n,m}$ on $\Lambda_{n,m}$ induced
by following expermient (``planted model'').
\begin{description}
\item[P1.] Generate a uniformly random $k$-partition $\sigma \in \brk{k}^n$.
\item[P2.] Generate a graph $G$ with $m$ edges chosen uniformly at random among the edges bicolored under $\sigma$.
\item[P3.] Output the pair $(G,\sigma)$.
\end{description} 
\begin{theorem}\label{Thm_ColorExchange}
Suppose that $d=2m/n<(2-\eps)k\ln k$.
There exists a function $f(n)=o(n)$ such that the following is true.
Let $\mathcal{D}$ be any graph property such that
$G(n,m)$ has $\mathcal{D}$ with probability $1-o(1)$,
and let $\EE$ be any property of pairs $(G,\sigma)\in\Lambda_{n,m}$.
If  for all sufficiently large $n$ we have
	\begin{equation}\label{eqColorExchange}
	\pr_{\mathcal{P}_{n,m}}\brk{(G,\sigma)\mbox{ has }\EE|G\mbox{ has }\mathcal{D}}
		\geq1-\exp(-f(n)),
	\end{equation}
then
	$\pr_{\mathcal{U}_{n,m}}\brk{(G,\sigma)\mbox{ has }\EE}=1-o(1).$
\end{theorem}

For a given assignment $\sigma\in\brk{k}^n$ we let $G(\sigma)$ be the set of all graphs
with $m$ edges for which $\sigma$ is a proper coloring.
Then it is immediate that
	$$|G(\sigma)|=\bink{\sum_{1\leq i<j\leq k}|\sigma^{-1}(i)|}{m}
		=\bink{(n^2-\sum_{i=1}^k|\sigma^{-1}(i)|^2)/2}{m}.$$
Hence,
	$$\lambda=\max_{\sigma}|G(\sigma)|\leq\bink{(1-1/k)\bink{n}2}{m}.$$

\begin{lemma}
There is a constant $\rho>0$ such that the following is true.
Let $\sigma\in\brk{k}^n$ be chosen uniformly at random.
Then $\pr\brk{|G(\sigma)|\geq\rho\lambda}\geq\rho.$
\end{lemma}
\begin{proof}
Let $\gamma>0$ be sufficiently small.
Moreover, let $n_i=|\sigma^{-1}(i)|$, $\delta_i=n_i-n/k$, and $N=(1-k^{-1})\bink{n}2$.
Then $\sum_i\delta_i=0$.
Therefore,
	$$\sum_in_i^2=\frac{n^2}k+\sum_i\delta_i^2.$$
Since for a random $\sigma$ the numbers $n_i$ are multinomially distributed,
with probability $\Omega(1)$ we have $|n_i-\frac{n}k|<\sqrt{\gamma n/k}$.
Hence, 
letting $N(\sigma)=\sum_{i<j}n_in_j$, we conclude that
there is a constant $\rho>0$ such that
	$$\pr\brk{N(\sigma)\geq N-\gamma n}\geq\rho.$$
Thus, by Stirling's formula with probability at least $\rho$ we have
	\begin{eqnarray*}
	\bink{N(\sigma)}m\bink{N}m^{-1}
		&\geq&\frac12\cdot\bcfr{N(\sigma)}{N}^m\bcfr{1-m/N}{1-m/N(\sigma)}^{N(\sigma)-m}(1-m/N)^{N(\sigma)-N}.
	\end{eqnarray*}
Since $N=\Omega(n^2)$ and $m=O(n)$, in the case $N(\sigma)\geq N-\gamma n$ we have
	\begin{eqnarray*}
	\bcfr{N(\sigma)}{N}^m&\geq&(1-\gamma n/N)^m\geq\Omega(1),\\
	\frac{1-m/N}{1-m/N(\sigma)}&\geq&1,\\
	(1-m/N)^{N(\sigma)-N}&\geq&\Omega(1).
	\end{eqnarray*}
Hence, choosing $\rho>0$ sufficiently small,
we can ensure that $\pr\brk{|G(\sigma)|\geq\rho\lambda}\geq\rho.$
\qed\end{proof}

\begin{corollary}\label{Cor_lambda}
We have $|\Lambda_{n,m}|\geq\rho^{2}k^n\lambda$.
\end{corollary}

\begin{lemma}\label{Lemma_plantedBound}
Suppose that there is a number $\zeta>0$ such that
	$\pr_{\mathcal{P}_{n,m}}\brk{\EE|\mathcal{C}}<\exp(-\zeta n).$
Then
	$$\abs{\cbc{(G,\sigma)\in\Lambda_{n,m}\cap\EE:G\in\mathcal{C}}}
		\leq\rho^{-2}\exp(-\zeta n)|\Lambda_{n,m}|.$$
\end{lemma}
\begin{proof}
We have
	\begin{eqnarray*}
	\exp(-\zeta n)&\geq&\pr_{\mathcal{P}_{n,m}}\brk{\EE|\mathcal{C}}
		=\frac{\pr_{\mathcal{P}_{n,m}}\brk{\EE\wedge\mathcal{C}}}{\pr_{\mathcal{P}_{n,m}}\brk{\mathcal{C}}}\\
		&=&k^{-n}\pr_{\mathcal{P}_{n,m}}\brk{\mathcal{C}}^{-1}
			\sum_{\sigma\in\brk{k}^n}\frac{\abs{\cbc{G\in G(\sigma):(G,\sigma)\in\EE\wedge G\in\mathcal{C})}}}{|G(\sigma)|}\\
		&\geq&\lambda^{-1}k^{-n}\pr_{\mathcal{P}_{n,m}}
			\abs{\cbc{(G,\sigma)\in\Lambda_{n,m}\cap\EE: G\in\mathcal{C})}},
	\end{eqnarray*}
because $|G(\sigma)|\leq\lambda$ for all $\sigma$.
Hence, Corollary~\ref{Cor_lambda} yields
	$$\exp(-\zeta n)\geq\frac{\rho^2}{|\Lambda_{n,m}|}\cdot \abs{\cbc{(G,\sigma)\in\Lambda_{n,m}\cap\EE: G\in\mathcal{C})}},$$
as desired.
\qed\end{proof}

Let $\mu=\Erw(|\SSS(\gnm)|)$ be the expected number of $k$-colorings of $\gnm$.
Combining the second moment argument from~\cite{kcol} with arguments from~\cite{AchFried},
we obtain the following result (see Appendix~\ref{Sec_FriedgutColor}).

\begin{lemma}\label{Lemma_FriedgutColor}
There is a function $f(n)=o(n)$ such that
$|\SSS(\gnm)|\geq\exp(-f(n))\mu$ with high probability.
\end{lemma}

\smallskip\noindent
\emph{Proof of Theorem~\ref{Thm_ColorExchange}.}
Assume that a random pair $(G,\sigma)$ chosen
according to the uniform model has $\EE$ with probability at least $2\zeta$,
while $\pr_{\mathcal{P}_{n,m}}\brk{\EE|\mathcal{C}}\leq\exp(-\zeta n)$ for an arbitrarily small $\zeta>0$.
Since $G_{n,m}$ has the property $\mathcal{C}$ \whp,
we conclude that
	$$\pr_{\mathcal{U}_{n,m}}\brk{\EE|\mathcal{C}}\geq\zeta.$$
Therefore, Lemma~\ref{Lemma_FriedgutColor} entails that
	$$b=\abs{\cbc{(G,\sigma)\in\Lambda_{n,m}\cap\EE:G\in\mathcal{C}}}
		\geq\zeta\exp(-f(n))|\Lambda_{n,m}|.$$
Since $f(n)=o(n)$, this contradicts Lemma~\ref{Lemma_plantedBound}.
\qed

\subsection{Proof of Lemma~\ref{Lemma_FriedgutColor}}\label{Sec_FriedgutColor}

To prove Lemma~\ref{Lemma_FriedgutColor}, we combine the second moment argument from~\cite{kcol}
with a sharp threshold argument.
Let $G=G(n,m)$ be a random graph and let $X$ be the number of \emph{balanced} colorings of $G$,
i.e., colorings $\sigma\in\brk{k}^n$ such that $|\sigma^{-1}(i)-n/k|\leq1$ for all $1\leq i\leq k$.
Recall that $\SSS(G)$ denotes the set of all $k$-colorings of $G$.
A direct computation involving Stirling's formula shows that
	$$\Erw(X)\geq\Omega(n^{-k/2})\Erw|\SSS(G)|.$$
In addition, \cite[Section~3]{kcol}  shows that there is a constant $C=C(k)$ such that
	$$\Erw(X^2)\leq C(k)\Erw(X)^2.$$
Applying the Paley-Zigmund inequality, we thus conclude that there is a number $\alpha=\alpha(k)>0$ such that
	$$\Pr\brk{X>\alpha\Erw(X)}\geq\alpha,$$
whence
	$$\Pr\brk{|\SSS(G)|\geq\Omega(n^{-k/2})\Erw|\SSS(G)|}\geq\alpha.$$
In addition, $\Erw|\SSS(G)|$ is easily computed: we have
	$$n^{-1}\ln\Erw|\SSS(G)|=\ln k+r\ln(1-k^{-1})+o(1),$$
where $r=m/n$.
Thus, we obtain the following.

\begin{lemma}\label{Lemma_FriedgutColor2ndmoment}
Let $\xi(k,r)=\ln k+r\ln(1-k^{-1})$.
Then $\pr\brk{n^{-1}\ln |\SSS(G(n,m))|\geq\xi(k,r)-o(1)}\geq\alpha$.
\end{lemma}

To complete the proof of Lemma~\ref{Lemma_FriedgutColor}, we combine Lemma~\ref{Lemma_FriedgutColor2ndmoment}
with a sharp threshold result.
Let $\mathcal{A}_\xi$ be the property that
a graph $G$ on $n$ vertices has less than $\xi^n$ $k$-colorings.

\begin{lemma}\label{Lemma_sharp}
For any fixed $\xi>0$ the property
$\mathcal{A}_\xi$ has a sharp threshold.
That is, there is a sequence $r_n$ such that for any $\eps>0$ we have
	$$\lim_{n\rightarrow\infty}\pr\brk{G(n,(1-\eps)\lceil r_nn\rceil)\mbox{ does not have }\mathcal{A}}
		=1-\lim_{n\rightarrow\infty}\pr\brk{G(n,(1+\eps)\lceil r_nn\rceil)\mbox{ has }\mathcal{A}}=0.$$
\end{lemma}
We shall prove Lemma~\ref{Lemma_sharp} in Appendix~\ref{Apx_sharp}.
Lemma~\ref{Lemma_FriedgutColor} is an immediate consequence of
Lemmas~\ref{Lemma_FriedgutColor2ndmoment} and~\ref{Lemma_sharp}.

\subsection{Proof of Lemma~\ref{Lemma_sharp}}\label{Apx_sharp}

The property $\mathcal{A}_\xi$ is monotone under the addition of edges.
Therefore, it is sufficient to prove that $\mathcal{A}_\xi$ has a sharp
threshold in the random graph $G(n,p)$, in which edges are added
with probability $p$ independently.
Let $\NNN=\xi^{n}$.
The argument builds upon~\cite{AchFried}.
We denote the set of $k$-colorings of a graph $G$ by $\SSS(G)$.

\begin{lemma}\label{Lemma_constraints}
Let $\NNN'$ be some number (that may depend on $n$), and let $0<t<1$ be fixed.
Further, let $M=O(1)$ be an integer.
Suppose that $\pr\brk{|\SSS(\gnp)|\leq\NNN'}\leq1-\tau$.
Moreover, assume that for a list of colors $c_1,\ldots,c_M\in\{1,\ldots,k\}$ the following is true:
if we pick $M$ vertices $v_1,\ldots,v_M$ indenpendently and uniformly at random, then with probability $\geq1-t/2$
the random graph $\gnp$ has at most $\NNN'$ $k$-colorings in which $v_i$ receives a color different from $c_i$
for all $1\leq i\leq M$.
Then with probability $\geq1-t/2+o(1)$ the random graph $\gnp$ has at most $\NNN'$ $k$-colorings
in which $v_i$ receives a color different from $c_i$ for all $1\leq i\leq M-1$.
\end{lemma}
\begin{proof}
Let $\omega$ be a (very) slow growing function of $n$.
Moreover, let $\EE_i$ be the event that the first $i$ constraints: ``$v_j$ must not receive color $c_j$'' for $1\leq j\leq i$,
cause the number of $k$-colorings to be at most $\NNN'$.
Then given that $G=\gnp$ has more than $\NNN'$ $k$-colorings (i.e., coditional on the event $\bar\EE_0$),
the probability of $\EE_M$ is at least $\frac12$.
Hence, conditional on $\bar\EE_0$, we have
	$$\pr\brk{\EE_{M-1}}+\pr\brk{\EE_M|\bar\EE_{M-1}}\pr\brk{\bar\EE_{M-1}}\geq\frac12.$$
We consider two cases: if $\pr\brk{\EE_{M-1}}\geq\frac12-\omega^{-1}$, then we are done.
Hence, assume that $\pr\brk{\EE_{M-1}}<\frac12-\omega^{-1}$.
Then $\pr\brk{\EE_M|\bar\EE_{M-1}}\geq\omega^{-1}$.
Note that $\pr\brk{\EE_M|\bar\EE_{M-1}}$ is the fraction of vertices $w$ such that forbidding color $c_M$ at $w$
reduces the number of colorings to at most $\NNN'$ (after the addition of the first $M-1$ constraints).
We call such vertices $w$ \emph{good}, and denote the set of colorings that $w$ spoils by $Z_w$ and its size by $z_w=|Z_w|$.
Let us consider two cases.
Let $y$ be the number of colorings of $G$ that respect the first $M-1$ constraints.
We consider two cases.
In each of these two cases we shall prove that adding a small number of random edges to $\gnp$
reduces the number of colorings that respect the first $M-1$ constraints to at most $\NNN'$ with probability at least $1-\omega^{-3}$.
\begin{description}
\item[Case 1: $y<\omega\NNN'$.]
	Since for each of the $y$ colorings the probability that a new random edge spoils this coloring is $k^{-1}$,
	we can reduce the number of colorings to at most $\NNN'$ by adding $\omega$ random edges (use Markov's inequality).
\item[Case 2: $y\geq\omega\NNN'$.]
	If $w,w'$ are good, then $|Z_w\cap Z_{w'}|\geq y-2\NNN'$.
	Therefore, adding an edge between two good vertices causes the number of colorings to drop to at most $2\NNN'$.
	Furthermore, the probability that a random edge joins two good vertices is $\pr\brk{\EE_M|\bar\EE_{M-1}}^2\geq\omega^{-2}$.
	Therefore, after adding $\omega^{10}$ edges, we have reduced the number of proper colorings to at most
	$2\NNN'$ with probability $\geq1-\omega^{-4}$.
	Finally, adding an additional $\omega^{10}$ edges reduces the number of colorings to at most $\NNN'$ by the same argument as in Case~1.
\end{description}
Now, note that instead of \emph{first} imposing the $M-1$ constraints $w_1,\ldots,w_{M-1}$ and \emph{then} adding the random edges
as in Cases~1 and~2 we could \emph{first} add a set of $2\omega^{10}$ random edges to $\gnp$.
As $\omega^{10}$ is of smaller order than the standard deviation of the number of edges of $\gnp$,
the resulting distribution is within $o(1)$ from the originial distribution $G(n,p)$ in total variation distance.
Therefore, we conclude that actually just imposing the $M-1$ constraints $w_1,\ldots,w_{M-1}$ suffices to increase the probability
of having $\leq\NNN'$ $k$-colorings to $1-\tau/2+o(1)$.
\qed\end{proof}

\begin{corollary}\label{Cor_constraints}
Let $\NNN'$ be some number (may depend on $n$), and let $0<t<1$ be fixed.
Further, let $M=O(1)$ be an integer.
Suppose that $\pr\brk{|\SSS(\gnp)|\leq\NNN'}\leq1-\tau$.
Then there is no list of colors $c_1,\ldots,c_M\in\{1,\ldots,k\}$ such that the following is true:
if we pick $M$ vertices $v_1,\ldots,v_M$ indenpendently and uniformly at random, then with probability $\geq1-t/2$
the random graph $\gnp$ has at most $\NNN'$ $k$-colorings in which $v_i$ receives a color different from $c_i$
for all $1\leq i\leq M$.
\end{corollary}
\begin{proof}
Applying the lemma $M$ times, we can reduce the number of constraints that is necessary to reduce the number
of colorings to $\leq\NNN'$ to $0$.
\qed\end{proof}

To prove that $\mathcal{A}_\xi$ has a sharp threshold, we assume for contratiction that this is not so.
Hence, there exists an edge probability $p^*$ such that the probability that $G_{n,p^*}$ has $\mathcal{A}_\xi$ is exactly equal to $1-t$
for a small $t>0$.
Further, by~\cite[Theorem~2.4]{EhudHunting}
there exists a fixed graph $R$ on $r$ vertices such that with probability $>1-t/3$ the following is true.
If we first pick $G=G_{n,p^*}$ and then insert a random copy of $R$ into $G$, then the resulting graph has $\mathcal{A}_\xi$.
Furthermore, this graph $R$ is $k$-colorable.
In fact, by monotonicity we may assume that $R$ is uniquely $k$-colorable.
The experiment of inserting a random copy of $R$ into $G_{n,p^*}$ is actually equivalent to the following (because $G_{n,p^*}$
is symmetric with respect to vertex permutations).
We let $G_R$ denote a random graph obtained by first inserting a copy of $R$ into the first $r$ vertices $v_1,\ldots,v_r$, and
then adding edges with probability $p^*$ independently (among all $n$ vertices $v_1,\ldots,v_n$).
Then the probability that $G_R$ has $\mathcal{A}_\xi$ is at least $1-t/3$.
Hence,
	\begin{equation}\label{eqAFWid1}
	\pr\brk{|\SSS(G_R)|\leq\NNN}\geq1-t/3,
	\end{equation}
while by the choice of $p^*$
	\begin{equation}\label{eqAFWid2}
	\pr\brk{|\SSS(G_{n,p^*})|>\NNN}\geq t>0.
	\end{equation}

Let $\hat G$ signify the subraph of $G_R$ induced on the $n-r$ vertices $v_{r+1},\ldots,v_n$.
Then $\hat G=G_{n-r,p^*}$, and~(\ref{eqAFWid2}) implies that
	\begin{equation}\label{eqAFWid2a}
	\pr\brk{|\SSS(\hat G)|>k^{-r}\NNN}\geq t.
	\end{equation}
Furthermore, we can relate the $k$-colorings of $G_R$ and the $k$-colorings of $\hat G$ as follows.
Let $Q$ be the set of edges from the set $\{v_1,\ldots,v_r\}$ to $\{v_{r+1},\ldots,v_n\}$.
Then \whp\ $|Q|=O(1)$ and no vertex in $\{v_{r+1},\ldots,v_n\}$ is incident to more than one edge in $Q$.
Furthermore, since $R$ admits a unique $k$-coloring, each edge in $Q$ forbids its endpoint in $\{v_{r+1},\ldots,v_n\}$ exactly
one color.
Hence, the edges in $Q$ impose constraints $c_1,\ldots,c_M$ on $M=|Q|$ randomly chosen vertices $w_1,\ldots,w_M$
as in Lemma~\ref{Lemma_constraints}.
Therefore, (\ref{eqAFWid1}) implies that
	$$\pr\brk{\hat G+M\mbox{ random constraints has at most $\NNN$ $k$-colorings}}\geq1-t/3.$$
But then Corollary~\ref{Cor_constraints} implies that
	$$\pr\brk{|\SSS(\hat G)|\leq \NNN}\geq1-t/3.$$
Furthermore, as we may add another $\ln n$ random edges to $\hat G$ without shifting the distribution by more than $o(1)$
in total variation distance, and since each of these $\ln n$ edges reduces the expected number of colorings by $k^{-1}$,
Markov's inequality entails that
	$$\pr\brk{|\SSS(\hat G)|\leq k^{-r}\NNN}\geq1-t/3-o(1),$$
which contradicts~(\ref{eqAFWid2a}).

\subsection{Proof of Theorem~\ref{thm:loose}}

Suppose that $d\leq (1-\eps)k\ln k$, and that $k\geq k_0(\eps)$ for a sufficiently large $k_0(\eps)$.
Let $q=5$ and recall that a graph is $\zeta$-choosable if for any assignments of color lists
of length at least $\zeta$ to the vertices of the graph there is a proper coloring
such that each vertex receives a color from its list.
To prove Theorem~\ref{thm:loose}, we consider the property
	$\EE$ that all vertices are loose and the following condition $\mathcal{D}$:
\begin{quote}
For any set $S\subset V$ of size $|S|\leq g(n)$ the subgraph
induced on $S$ is $(q-1)$-choosable.%
\end{quote}
Here $q>0$ is a constant and $g(n)=\sqrt{nf(n)}=o(n)$, where $f(n)$ is
the function from Theorem~\ref{Thm_ColorExchange}.

\begin{lemma}\label{Lemma_choosable}
With high probability the random graph $G(n,m)$ satisfies $\mathcal{D}$.
\end{lemma}
\begin{proof}
Since $m=O(n)$,
this follows from a standard first moment argument.
\qed\end{proof}
	
By Theorem~\ref{Thm_ColorExchange}, we just need to establish~(\ref{eqColorExchange}).
Thus, let $\sigma\in\brk{k}^n$ be a coloring such that the color classes $V_i=\sigma^{-1}(i)$
satisfy $|V_i|\sim n/k$, and let $G$ be a random graph with $m$ edges such that $\sigma$
is a $k$-coloring of $G$.
Let $v_0\in V$ be any vertex; without loss of generality we may assume that $\sigma(v_0)=1$.
In addition, let $1<l\leq k$ be the ``target color'' for $v_0$.
If $v_0$ has no neighbor in $V_l$, then we can just assign this color to $v_0$.

Otherwise, we run the following process.
In the course of the process, every vertex is either \emph{awake}, \emph{dead}, or \emph{asleep}.
Initially, all the neighbors of $v_0$ in $V_l$ are awake, $v$ is dead, and all other vertices are asleep.
In each step of the process, pick an awake vertex $w$ arbitrarily and declare it dead
(if there is no awake vertex, terminate the process).
If there are at least $q$ colors $c_1(w),\ldots,c_q(w)$ such that
$w$ has no neighbor in $V_{c_i(w)}$, then we do nothing.
Otherwise, we pick $q$ colors $c_1(w),\ldots,c_q(w)$ randomly and
declare all asleep neighbors of $w$ in $V_{c_j(w)}$ awake for $1\leq j\leq q$.

\begin{lemma}\label{Lemma_fewDead}
With probability at least $1-\exp(-f(n))$ there are
at most $g(n)$ dead vertices when the process terminates.
\end{lemma}
\begin{proof}
We show that the aforementioned process is dominated by a branching process
in which the expected number of successors is less than one.
Then the assertion follows from Chernoff bounds.

To set up the analogy, note that the expected number of neighbors of any $w\in V\setminus V_i$
in $V_i$ is asymptotically $2m/k<\frac{1-\eps}{1-k}\cdot\ln k$.
Hence, the probability that $w$ has no neighbor in $V_i$ is at least $k^{\eps/2-1}$.
Therefore, the expected number of classes $i\not=\sigma(w)$ in which
$w$ has no neighbor is at least $(k-1)k^{\eps/2-1}\geq k^{\eps/3}$.
Furthermore, the number of such classes is asymptotically binomially distributed.
Therefore, assuming that $k$ is sufficiently large, we conclude that
the probability that there are less than $q$ classes in which $w$ has no neighbor
is less than $k^{-1}$.
Given that this is so, the number of neighbors of $w$ in each of the $q$ chosen
classes $c_1(w),\ldots,c_q(w)$ has mean at most $2\ln k$.
Therefore, the expected number of newly awake vertices resulting from $w$
is at most $2k^{-1}\ln k$.
Thus, the above process is dominated by a branching process with successor
rate $2k^{-1}\ln k<1$.
Therefore, the assertion follows from stochastic domiance and Chernoff bounds.
\qed\end{proof}

\smallskip\noindent
\emph{Proof of Theorem~\ref{thm:loose}.}
Let $S$ be the set of dead vertices left by the aforementioned process.
By Lemma~\ref{Lemma_fewDead} we may assume that $|W|\leq g(n)$.
Hence, conditioning on $\mathcal{D}$,
we may assume that $S$ is $(q-1)$-choosable.
Now, we assign lists of colors to the vertices in $S$ as follows.
The list of $v_0$ just consists of its target color $l$.
To any other $w\in W$ we assign the list $L_w=\{c_1(w),\ldots,c_q(w)\}\setminus\{l\}$.
Now, we can color the subgraph $G\brk{S}$ by assigning color $l$ to $v$
and a color from $L_w$ to any other $w\in W$.
We extend this to a coloring of $G$ by assigning color $\sigma(u)$ to
any $u\in V\setminus W$.
Let $\tau$ signify the resulting coloring.

We claim that $\tau$ is a proper coloring of $G$.
For both the subgraph induced on $W$ and the subgraph induced on $V\setminus W$
are properly colored.
Moreover, by construction no $w\in W\setminus\{v\}$ is adjacent to a vertex of color $c_j(w)$ in $V\setminus W$.
Finally, $\sigma$ and $\tau$ are at Hamming distance at most $|W|\leq g(n)=o(n)$.
Hence, the assertion follows from Theorem~\ref{Thm_ColorExchange}.
\qed

\subsection{Rigid variables}

Let $\alpha,\eps>0$, and assume that $k\geq k_0(\eps)$ for a large enough $k_0(\eps,\alpha)>0$.
Suppose that $(1+\eps)k\ln k\leq d=2m/n\leq(2-\eps)k\ln k$.
To prove Theorem~\ref{thm:froz} for coloring, we apply Theorem~\ref{Thm_ColorExchange} as follows. 
We let $\beta=\beta(\eps,\alpha)>0$ be a sufficiently small number and denote by
$\EE$ the following property of  a pair $(G,\sigma)\in\Lambda_{n,m}$.
	\begin{equation}\label{eqG*}
	\parbox{15cm}{
	There is a subgraph $G_*\subset G$ of size $|V(G_*)|\geq(1-\alpha)n$
	such that for every vertex $v$ of $G_*$ and each color $i\not=\sigma(v)$ there are
	at least $\beta\ln k$ vertices vertex $w$ in $G_*$ that are adjacent to $v$
	such that $\sigma(w)=i$.}
	\end{equation}
Also, we let $\mathcal{D}$ be the property that the maximum degree is at most $\ln^2n$.

\begin{lemma}\label{Lemma_core}
Condition~(\ref{eqColorExchange}) is satisfied with $\mathcal{D}$ and $\EE$ as above.
\end{lemma}

\medskip
\noindent\emph{Proof of Theorem~\ref{thm:froz} for coloring.}
Given a random coloring $\sigma$ of a random graph $G=G(n,m)$,
Lemma~\ref{Lemma_core} and Theorem~\ref{Thm_ColorExchange} imply that
\whp\ there is  a subgraph $G_*$ satisfying~(\ref{eqG*}).
In addition, we assume that $G$ has the following property.
	\begin{equation}\label{eqGsparse}
	\parbox{15cm}{
	There is no set $S\subset V$ of size $|S|\leq n/(k\ln k)$ that spans at least
	$|S|\frac{\beta}2\ln k$ edges.}
	\end{equation}
A standard 1st moment argument shows that (\ref{eqGsparse})  holds in $G(n,m)$ \whp\

Assume  for contradiction
that there is another coloring $\tau$
such that the set $U=\{v\in G_*:\sigma(v)\not=\tau(v)\}$ has size $|U|\leq  n/(k\ln k)$.
Let $U_i^+=\{v\in G_*:\tau(v)=i\not=\sigma(v)\}$
and $U_i^-=\{v\in G_*:\sigma(v)=i\not=\tau(v)\}$.
Then
	\begin{equation}\label{eqrigid1}
	|U|=\sum_{i=1}^k|U_i^+|=\sum_{i=1}^k|U_i^-|.
	\end{equation}

%Let $\mathcal{P}$ be the following graph property: there is no set $S\subset V$ of size $|S|\leq n/(k\ln k)$ such that
%$e(S)\geq|S|\frac{\delta}2\ln k$, where $\delta$ is some small but positive constant (and we assume that $k$ is sufficiently large).
%A standard 1st moment argument shows that $\mathcal{P}$ holds in $G_{n,m}$ \whp\

Every $v\in G_*\setminus V_i$ has at least $\beta\ln k$ neighbors in $G_*\cap \sigma^{-1}(i)$.
Hence, if $v\in U_i^+$, then all of these neighbors lie inside of $U_i^-$.
%	$$\forall v\in U_i^+:N_{G_*}(v)\cap V_i\subset U_i^-.$$
We claim that this implies that $|U_i^+|<|U_i^-|$.
For assume that $U_i^+\geq U_i^-$ and set $S=U_i^+\cup U_i^-$.
Then $|S|\leq |U|\leq n/(k\ln k)$, and $S$ spans at least $|S|\frac{\beta}2\ln k$ edges,
in contradiction to~(\ref{eqGsparse}).
Thus, we conclude that $|U_i^+|<|U_i^-|$ for \emph{all} $i$, in contradiction to~(\ref{eqrigid1}).
Hence, all the vertices in $G_*$ are $\Omega(n)$-rigid.
\qed

%
%We refer to $G_*$ as the \textbf{core} of $(G,\sigma)$.
%Clearly, $G_*$ is contained in the maximal subraph of minimum degree $k$  of $G$.
%In fact, the construction ensures that each $v\in V(G_*)$ has at least $\frac{\eps}2\ln k\geq 1$ neighbors
%of each color $i\not=\sigma(v)$ in $G_*$.
%Thus, 

%\begin{remark}
%Since each vertex $v$ of $G_*$ has actually not just one but
%at least $\frac{\eps}2\ln k$ neighbors of every color $i\not=\sigma(v)$ inside of $G_*$,
%one can show that $G_*$ has no $k$-coloring $\tau\not=\sigma$ within Hamming distance at most $n/k\ln k$ of $\sigma$.
%Indeed, this stronger property extends to the uniform model.
%Hence, in the uniform model it is typical that $(G,\sigma)$ features a subgraph $G_*$ of size at least $(1-\alpha)n$ that has no coloring
%$\tau\not=\sigma$ such that $\dist(\sigma,\tau)<n/(k\ln k)$.
%In other words, in order to recolor $G_*$ while staying close to $\sigma$, essentially the best we can do is permuting entire
%color classes.
%\end{remark}

\subsection{Proof of Lemma~\ref{Lemma_core}}\label{Apx_core}

Let $(G,\sigma)\in\Lambda_{n,m}$ be a random pair chosen from the distribution $\mathcal{P}_{n,m}$.
We may assume that $|\sigma^{-1}(i)|\sim n/k$ for all $i$ and let $V_i=\sigma^{-1}(i)$.
To simplify the analysis, we shall replace the random graph $G$, which has a \emph{fixed} number $m$ of edges,
by a random graph $G'$ in which 
is obtained by including each edge $\{v,w\}$ with $\sigma(v)\not=\sigma(w)$ with probability $p$
independently.
Here $p$ is chosen so that the \emph{expected} number
	$\sum_{1\leq i<j\leq k}|V_i|\cdot|V_j|\cdot p$
of edges of $G'$ equals $m$.

\begin{lemma}\label{Lemma_independentEdges}
For any property $\mathcal{Q}$ we have
$\pr\brk{G\mbox{ has }\mathcal{Q}|\mathcal{D}}\leq O(\sqrt{n}\cdot \pr\brk{G'\mbox{ has }\mathcal{Q}|\mathcal{D}})$.
\end{lemma}
We defer the proof to Section~\ref{Sec_independentEdges}.

Thus, in the sequel we will work with $G'$ rather than $G$.
Let $\gamma=\gamma(\eps)>0$ be a sufficiently small number, and
let $V_i=\sigma^{-1}(i)$.
Moreover, for a vertex $v$ and a set $Z\subset V$ let $e(v,Z)$ signify the number of $v$-$Z$-edges in $G'$.
We construct a subgraph $G_*$ of $G'$ as follows.
%\amin{We could say that this is reminiscent of Alon+Kahale or of the construction of a $k$-core.}
\begin{enumerate}
\item Let $W_{ij}=\{v\in V_i:e(v,V_j)<\gamma\ln k\}$, $W_i=\bigcup_{j=1}^kW_{ij}$, and $W=\bigcup_{i=1}^kW_i$.
\item Let $U_{il}=\{v\in V_l:e(v,W_i)>\frac{\gamma}2\ln k\}$ and $U=\bigcup_{i\not=l}U_{il}$.
\item Let $Z=U$.
	While there is a vertex $v\in V\setminus Z$ that has at least $10$
	neighbors in $Z$, add $v$ to $Z$.
\item Let $G_*=G'-\bigcup_{i=1}^kW_i-Z$.
\end{enumerate}

For each vertex $v\in V_i$ and each color $j\not=i$ the expected number of neighbors of $v$ with color $i$
is $|V_i|\cdot\frac{2m}n\sim(1+2\eps)\ln k$.
Hence, the sets $W_{ij}$ contain those vertices form $V_i$ that have a lot fewer neighbors with color $j$ than expected.

\begin{lemma}\label{Lemma_Wijbound}
There is a number $\beta=\beta(\eps)>0$ such that
with probability $\geq1-\exp(-\Omega(n))$ we have $|W_{ij}|<nk^{-2-\beta}$ for any $i,j$.
Hence, $|W_i|\leq n k^{-1-\beta}$, and $|W|\leq n k^{-\beta}$.
\end{lemma}
\begin{proof}
In the random graph $G'$ for each $v\in V_i$ the number $e(v,V_j)$ is binomially distributed.
Hence, the probability that $e(v,V_j)<\gamma\ln k$ is at most $\exp(-(1+\eps')\ln k)$, where $\eps'>0$ depends only on $\eps$ and $\gamma$.
Furthermore, 
as  in $G'$ edges occur independently,
$|W_{ij}|$ is binomially distributed as well (with mean $\leq \frac{n}k\cdot\exp(-(1+\eps')\ln k)$).
Therefore, the assertion follows from Chernoff bounds.
\qed\end{proof}

Each of the vertices in $U$ has \emph{a lot} of neighbors in the small set $W$.
Therefore, since the random graph $G'$ is a good expander, we expect
$U$ to be much smaller than $W$.

\begin{lemma}\label{Lemma_Ubound}
Given that $\mathcal{D}$ occurs, with probability at least $1-\exp(-\Omega(n))$ the set
$U$ contains at most $nk^{-7}$ vertices.
\end{lemma}
We postpone the proof to Section~\ref{Sec_Ubound}.

\begin{lemma}\label{Lemma_Zbound}
With probability $\geq1-\exp(-\Omega(n))$
the set $Z$ contains at most $nk^{-6}$ vertices.
\end{lemma}
\begin{proof}
Assume that this is not the case.
Let $Y$ contain all vertices of $U$ and the first $nk^{-6}-|U|$ vertices added
to $Z$ by step~3 of the construction of $G_*$.
Then $|Y|\leq n k^{-6}$ and $e(Y)\geq9|Y|$.
However,  a simple first moment argument shows that
the probability that a set $Y$ with these two properties is present in $G'$ is at most $\leq\exp(-\Omega(n))$.
\qed\end{proof}

Combining Lemma~\ref{Lemma_Wijbound}, \ref{Lemma_Ubound}, and~\ref{Lemma_Zbound}, 
we conclude that $G_*$ contains at least $n(1-\alpha)$ vertices (provided that $k$ is sufficiently larger).
Moreover, the construction of $G_*$ ensures that this graph satisfies~(\ref{eqG*}).

\subsection{Proof of Lemma~\ref{Lemma_independentEdges}}\label{Sec_independentEdges}

Given that $G'$ has exactly $m$ edges, $G'$ is just a uniformly random graph with planted coloring $V_1,\ldots,V_k$.
That is, given that the number of edges is $m$, $G'$ is identically distributed to $G$.
Therefore,
	\begin{eqnarray}\nonumber
	\pr\brk{G'\mbox{ has }\mathcal{P}|\mathcal{D}}
		&\geq&
		\frac{\pr\brk{G'\mbox{ has both }\mathcal{P}\mbox{ and }\mathcal{D}\mbox{ and has $m$ edges}}}{\pr\brk{G'\mbox{ has }\mathcal{D}}}\\
		&\geq&\Omega(n^{-\frac12})\cdot
			\frac{\pr\brk{G\mbox{ has both }\mathcal{P}\mbox{ and }\mathcal{D}}}{\pr\brk{G'\mbox{ has }\mathcal{D}}}.
	\label{eqBinomEx1}
	\end{eqnarray}
Furthermore, since $m=O(n)$,
with probability $1-o(1)$ the maximum degree of $G$ as well as of $G'$ is at most $\ln n$.
Therefore, $G,G'$ have $\mathcal{D}$ with probability $1-o(1)$.
Hence, (\ref{eqBinomEx1}) yields
	\begin{eqnarray*}\nonumber
	\pr\brk{G'\mbox{ has }\mathcal{P}|\mathcal{D}}
		&\geq&\Omega(n^{-\frac12})\cdot
			\frac{\pr\brk{G\mbox{ has both }\mathcal{P}\mbox{ and }\mathcal{D}}}{\pr\brk{G\mbox{ has }\mathcal{D}}}
		=\Omega(n^{-\frac12})\cdot\pr\brk{G\mbox{ has }\mathcal{P}|\mathcal{D}},
	\label{eqBinomEx2}
	\end{eqnarray*}
as claimed.

\subsection{Proof of Lemma~\ref{Lemma_Ubound}}\label{Sec_Ubound}

To analyze the sets $U_{il}$ from the second step of the construction of $G_*$, consider
	\begin{eqnarray*}
	U_{il}'&=&\{v\in V_l:e(v,W_i\setminus W_{il})>\frac{\gamma}4\ln k\},\\
	U_{il}''&=&\{v\in V_l:e(v,W_{il})>\frac{\gamma}4\ln k\}.
	\end{eqnarray*}

\begin{lemma}
With probability $\geq1-\exp(-\Omega(n))$ we have $|U_{il}'|\leq nk^{-10}$.
\end{lemma}
\begin{proof}
The definition of the set $W_i\setminus W_{il}$ depends solely on the $V_i$-$V\setminus V_l$-edges.
Therefore, the $V_l$-$V_i$-edges are indepenent of the random set $W_i\setminus W_{il}$,
which with probability $\geq1-\exp(-\Omega(n))$ has size $\leq n k^{-1-\beta}$ by the Lemma~\ref{Lemma_Wijbound}.
Assuming that this is indeed the case, we conclude that for any vertex $v\in V_l$ the number $e(v,W_i\setminus W_{il})$ is
binomially distributed with mean $np k^{-1-\beta}\leq2k^{-\beta}\ln k$.
Hence, the probability that $v$ has $z=\frac{\gamma}4\ln k$ neighbors inside $W_i\setminus W_{il}$ is at most
	\begin{eqnarray*}
	\bink{n k^{-1-\beta}}zp^z&\leq&\bcfr{8e}{\gamma k^{\beta}}^z\leq k^{-10}/2.
	\end{eqnarray*}
Thus, $\Erw|U_{il}'|\leq nk^{-10}/2$.
Finally, as $|U_{il}'|$ is binomially distributed, the assertion follows from Chernoff bounds.
\qed\end{proof}

%\emph{Due to the size of the $U_{il}'$-sets, the best dependence between $k$ and $\eps$ that
%this construction gives is about $\eps=C/\ln\ln k$ for some constant $C>0$.
%Possible a refined construction would give $\eps=C\ln\ln k/\ln k$.}

%Let $\mathcal{D}$ be the event that the maxium degree of $G$ is $\leq\ln^2n$.

\begin{lemma}
Conditional on the event $\mathcal{D}$, with probability $\geq1-\exp(-\Omega(n))$ we have $|U_{il}''|\leq nk^{-10}$.
\end{lemma}
\begin{proof}
We just need to analyze the bipartite subgraph $G'\brk{V_i\cup V_l}$.
The set $W_{il}$ consists of all vertices $v\in V_i$ that have degree $<\gamma\ln k$ in this subgraph.
To investigate $G'\brk{V_i\cup V_l}$, we condition on the degree sequence $\vec d$ of this bipartite graph.
Since we also condition on the event $\mathcal{D}$, the maximum degree is $\leq\ln^2n$.
Hence, we can generate the random bipartite graph with degree sequence $\vec d$ via the configuration model, and the probability
that the resulting multigraph happens to be a simple graph is $\geq\exp(-O(\ln^4n))$.
Thus, we just need to study a random configuration.

Now, in a random configuration the probability that a vertex $v\in V_l$ has $\frac{\gamma}4\ln k$ neighbors in the
set $W_{il}$ is $\leq k^{-10}$, because the total number of edges of $G'\brk{V_i\cup V_l}$ is concentrated about $n^2pk^{-2}$.
Therefore, the (conditional) expected size of $U_{il}''$ is $\leq n k^{-11}$.
Consequently, Azuma's inequality yields that with probability $\geq1-\exp(-\Omega(n))$ the size of
$U_{il}''$ is $\leq nk^{-10}$.
\qed\end{proof}

Finally, Lemma~\ref{Lemma_Ubound} follows immedately from the fact that 
$U_{il}\subset U_{il}'\cup U_{il}''$.

\subsection{Proof of Theorem~\ref{Thm_icy_col}}\label{Sec_icy_col}

To prove the coloring part of Theorem~\ref{Thm_icy_col},
we need to come up with an appropriate way to measure how ``similar'' two $k$-colorings of a given graph are $G=G(n,m)$.
A first idea may be to just use the Hamming distance.
However, if we construct
a coloring $\tau$ simply by permuting the color classes of another coloring $\sigma$,
then $\sigma$ and $\tau$ can have Hamming distance $n$, although they are essentially identical.
Therefore, instead of the Hamming distance
we shall use the following concept.
Given two coloring $\sigma,\tau$, we let $M_{\sigma,\tau}=(M^{ij}_{\sigma,\tau})_{1\leq i,j\leq k}$ be the matrix with entries
	$$M^{ij}_{\sigma,\tau}=n^{-1}|\sigma^{-1}(i)\cap\tau^{-1}(j)|.$$
%In words, the $ij$-entry of $M_{\sigma,\tau}$ equals the proportion of vertices that have color $i$ in $\sigma$ and color $j$ in $\tau$.
Then to measure how close $\tau$ is to $\sigma$ we let
	$$f_{\sigma}(\tau)=\|M_{\sigma,\tau}\|_F^2=\sum_{i,j=1}^k(M^{ij}_{\sigma,\tau})^2$$
be the squared Frobenius norm of $M_{\sigma,\tau}$.
Hence, $f_{\sigma}$ is a map from the set $\brk{k}^n$ of $k$-colorings to the interval $\brk{k^{-2},f_{\sigma}(\sigma)}$,
where $f_{\sigma}(\sigma)\geq k^{-1}$.
(Thus, the \emph{larger} $f_{\sigma}(\tau)$, the more $\tau$ resembles $\sigma$.)
Furthermore, for a fixed $\sigma\in\SSS(G)$ and a number $\lambda>0$ we let
	$$g_{\sigma,G,\lambda}(x)=|\{\tau\in[k]^n:f_\sigma(\tau)=x\wedge H(\tau)\leq\lambda n\}|.$$

In order to show that $\SSS(\gnm)$ with $m=rn$ decomposes into exponentially many regions, we employ the following lemma.

\begin{lemma}\label{Lemma_cluster_col}
Suppose that $r>(\frac12+\eps_k)k\ln k$.
There are numbers $k^{-2}<y_1<y_2<k^{-1}$ and $\lambda,\gamma>0$ such that
with high probability 
 a pair $(G,\sigma)\in\Lambda_{n,m}$ chosen from the distributoin $\mathcal{U}_{n,m}$
 has the following two properties.
\begin{enumerate}
\item For all $x\in\brk{y_1,y_2}$ we have $g_{\sigma,G,\lambda}(x)=0$.
\item The number of colorings $\tau\in\SSS(G)$ such that $f_{\sigma}(\tau)>y_2$ is at most $\exp(-\gamma n)\cdot|\SSS(G)|$.
\end{enumerate}
\end{lemma}
Let $G=\gnm$ be a random graph and call $\sigma\in\SSS(G)$ \emph{good} if 1.\ and 2.\ hold.
Then Lemma~\ref{Lemma_cluster_col} states that with high probability
a $1-o(1)$-fraction of all $\sigma\in\SSS(G)$ is good.
Hence,  to decompose $\SSS(G)$ into regions, we proceed as follows.
For each $\sigma\in \SSS(G)$ we let
	$$\mathcal{C}_\sigma=\{\tau\in\SSS(G):f_\sigma(\tau)>y_2\}.$$
Then starting with the set $S=\SSS(G)$ and removing iteratively some $\mathcal{C}_\sigma$ for a good $\sigma\in S$ from
$S$ yields an exponential number of regions.
Furthermore, each such region $\mathcal{C}_\sigma$ is separated by a linear Hamming distance
from the set $\SSS(G)\setminus \mathcal{C}_\sigma$, because $f_\sigma$ is continuous with respect to $n^{-1}\times$Hamming distance
(that is, for any $\eps>0$ there is $\delta>0$ such that $f_\sigma(\tau)<\eps$ for all $\tau\in\brk{k}^n$ satisfying
	$\dist(\sigma,\tau)<\delta n$).
Thus, the property stated in Theorem~\ref{Thm_icy_col} follows from Lemma~\ref{Lemma_cluster_col}.
%In fact, any path between $\mathcal{C}_\sigma$ and $\SSS(G)\setminus\mathcal{C}_\sigma$ has height at least $\lambda n$.

To establish Lemma~\ref{Lemma_cluster_col}, we employ the planted model.

\begin{lemma}\label{Lemma_cluster_col_pl}
Suppose that $r>(\frac12+\eps_k)k\ln k$.
There are $k^{-2}<y_1<y_2<k^{-1}$ and $\lambda,\gamma>0$ such that
with probability at least $1-\exp(-\Omega(n))$
a pair $(G,\sigma)\in\Lambda_{n,m}$ chosen from the distributoin $\mathcal{P}_{n,m}$
the two properties stated in Lemma~\ref{Lemma_cluster_col}.
\end{lemma}
Thus, Lemma~\ref{Lemma_cluster_col} follows from Lemma~\ref{Lemma_cluster_col_pl} and Theorem~\ref{Thm_ColorExchange}.

\medskip\noindent\emph{Proof of Lemma~\ref{Lemma_cluster_col_pl}.}
The proof is based on the first moment method.
Let $\sigma\in\brk{k}^n$ be a fixed assignment of colors to the vertices.
We may assume that $|\sigma^{-1}(i)|\sim n/k$ for all $1\leq i\leq k$, because all but an exponentially small fraction
of all assignments in $\brk{k}^n$ have this property.
Further, let $G$ be a graph with $m$ edges such that $\sigma$ is a $k$-coloring of $G$ chosen uniformly at random from the set
of all such graphs.
A direct computation shows that for an assignment $\tau\in\brk{k}^n$ the probability
that $H(\tau)\leq\lambda n$ is	
	\begin{equation}\label{eqColorFirst1}
	\leq\bcfr{1-2k^{-1}+f_\sigma(\tau)}{1-k^{-1}}^{rn}\exp((\psi(\lambda)+o(1))n),
	\end{equation}
where $\lim_{\lambda\rightarrow0}\psi(\lambda)=0$.
To prove the lemma, we shall compute the \emph{expected} number of assignments $\tau$
such that $H(\tau)\leq\lambda n$ and $f_{\sigma}(\tau)=x$ for a suitable $y_1<x<y_2$.

To this end, we have to take into account the number of possible colorings $\tau$.
We parameterize the set of all possible $\tau$ by a matrix $A=(a_{ij})_{1\leq i,j\leq k}$,
where $a_{ij}=n^{-1}|\sigma^{-1}(i)\cap\tau^{-1}(j)|$.
Then by~(\ref{eqColorFirst1}) the contribution of a matrix $A$ to the first moment is at most
	$$\mathcal{F}(A)=k^{-n}\bink{n}{(a_{ij}n)_{1\leq i,j\leq k}}\bcfr{1-2k^{-1}+f_\sigma(\tau)}{1-k^{-1}}^{rn}\exp((\psi(\lambda)+o(1))n)$$
(the $k^{-n}$ accounts for the fact that we consider the coloring $\sigma$ fixed).
Taking logarithms, we obtain
	$$n^{-1}\ln\mathcal{F}(A)
		\sim-\ln k-\sum_{i,j=1}^ka_{ij}\ln(a_{ij})+r\ln\bcfr{1-2k^{-1}+\sum_{i,j=1}^ka_{ij}^2}{1-k^{-1}}+\psi(\lambda).$$
For a given number $x$ we let
	$\mathcal{A}(x)$ be the set of all matrices $A=(a_{ij})_{1\leq i,j\leq k}$
	such that $a_{ij}\geq0$, $\sum_{i=1}^ka_{ij}\sim k^{-1}$, and $\sum_{i,j=1}^ka_{ij}^2=x$.
Since there are at most $n^{k^2}$ possible matrices $A$, for any given $x$
the expected number of colorings
$\tau$ such that $f_{\sigma}(\tau)=x$ is at most
	$$n^{k^2}\max_{A\in\mathcal{A}(x)}\mathcal{F}(A).$$
Hence, by continuity it suffices to show that for some $x\in\brk{y_1,y_2}$ the expression
	$n^{-1}\max_{A\in\mathcal{A}(h)}\ln\mathcal{F}(A)$
is strictly negative for a small enough $\lambda>0$.

Let $h=k^{-3/2}$ and $x=k^{-1}-2h$.
Then Theorem~9 from~\cite{kcol} shows that the maximum $\max_{A\in\mathcal{A}(x)}\ln\mathcal{F}(A)$
is attained for a matrix $A$ with entries
	$$a_{ii}=k^{-1}-h+o(h),\quad a_{ij}\sim h(k-1)^{-1}\quad(i\not=j)$$
asymptotically as $k$ grows.
An explicit computation shows that for this matrix $A$ the value $\ln\mathcal{F}(A)$ is strictly negative,
provided that $\lambda$ is sufficiently small.
Therefore, 
we can apply Markov's inequality to complete the proof.
\qed

\section{Proofs for Random $k$-SAT}

\subsection{The planted model}

Consider the distribution $\mathcal{U}_{n,m}$ on the set
$\Lambda_{n,m}$ of pairs $(F,\sigma)$, where $F$ is a $k$-SAT formula
with variables $x_1,\ldots,x_n$ and with $m$ clauses, and $\sigma$ is a satisfying assignment of $F$.
\begin{description}
\item[U1.] Generate a random formula $F=F_k(n,m)$.
\item[U2.] Sample a satisfying assignment $\sigma$ of $F$ uniformly at random;
	if $F$ is unsatisfiable, fail.
\item[U3.] Output the pair $(F,\sigma)$.
\end{description}
To analyze this distribution, we consider the distribution $\mathcal{P}_{n,m}$ on $\Lambda_{n,m}$ induced
by following expermient.
\begin{description}
\item[P1.] Generate a random assignment $\sigma\in\{0,1\}^n$.
\item[P2.] Generate a random $k$-CNF formula $F$ with $m$ clauses chosen uniformly among those satisfied by $\sigma$.
\item[P3.] Output the pair $(F,\sigma)$.
\end{description}

Our goal is to establish the following connection between these two distributions.

\begin{theorem}\label{Thm_SATExchange}
There is a sequence $\eps_k\rightarrow0$ such that the following holds.
Let $m=\lceil rn\rceil$ for some $r<(1-\eps_k)2^k\ln 2$,
and let $f(n)$ be any function that such that $\lim_{n\rightarrow\infty}f(n)=\infty$.
Let $\mathcal{D}$ be any property such that
$F_k(n,m)$ has $\mathcal{D}$ with probability $1-o(1)$,
and let $\EE$ be any property of pairs $(F,\sigma)\in\Lambda_{n,m}$.
If  for all sufficiently large $n$ we have
	\begin{equation}\label{eqSATExchange}
	\pr_{\mathcal{P}_{n,m}}\brk{(F,\sigma)\mbox{ has }\EE|F\mbox{ has }\mathcal{D}}
		\geq1-\exp(-k2^{3-k}n-f(n)),
	\end{equation}
then
	$\pr_{\mathcal{U}_{n,m}}\brk{(F,\sigma)\mbox{ has }\EE}=1-o(1).$
\end{theorem}

The proof of Theorem~\ref{Thm_SATExchange} is based on the following lemma,
which we establish in Section~\ref{Sec_numberkSAT}.

\begin{lemma}\label{Lemma_numberkSAT}
Let $\mu=2^n(1-2^{-k})^m$ denote the expected number of satisfying
assignments of a random $k$-CNF $F_k(n,m)$.
Then for $k\geq 8$, \whp
	$$|\SSS(F_k(n,m)|\geq\mu\exp(-k2^{3-k}n) \enspace .$$ 
\end{lemma}

\begin{proof}[Theorem~\ref{Thm_SATExchange}]
Assume for contradiction that there is a fixed $\alpha>0$ such that
$\pr_{\mathcal{U}_{n,m}}\brk{(F,\sigma)\mbox{ has }\EE}\geq\alpha$ for infinitely many $n$.
Then Lemma~\ref{Lemma_numberkSAT} implies that the set
$L=\Lambda_{n,m}\setminus\EE$ has size
	\begin{equation}\label{eqSATEx1}
	|L|\geq\frac{\alpha}2\mu\exp(-k2^{3-k}n)\brk{2\bink{n}k}^m=\frac{\alpha}2\exp(-k2^{3-k}n)|\Lambda_{n,m}|.
	\end{equation}
On the other hand, as $\mathcal{P}_{n,m}$ is just the uniform distribution on the set $\Lambda_{n,m}$,
(\ref{eqSATExchange}) implies that
	$$|L|\leq\exp(-k2^{3-k}n-f(n))|\Lambda_{n,m}|.$$
As $f(n)\rightarrow\infty$, this contradicts~(\ref{eqSATEx1}) for sufficiently large $n$.
\qed\end{proof}

\subsection{Proof of Lemma~\ref{Lemma_numberkSAT}}\label{Sec_numberkSAT}

Let $\Lambda_b$ be the function defined by
%from~\cite{AP} we have 
\begin{equation}
\label{eq:b}
\Lambda_b(1/2,k,r) = 4\left[\frac{\big((1-\epsilon/2)^k-2^{-k}\big)^2}{(1-\epsilon)^k}\right]^r
\enspace ,
\end{equation}
where $\epsilon$ satisfies
\begin{equation}\label{epdef}
\epsilon (2-\epsilon)^{k-1} =1 \enspace . 
\end{equation} 

\begin{lemma}\label{Thm_Amin}
Suppose that $r < 2^k \ln 2 - k$.
Then $F_k(n,rn)$ has at least $(\Lambda_b(1/2,k,r)-o(1))^{n/2}$
satisfying assignments w.h.p.
\end{lemma}

Recall that $F_k(n,m)$ denotes a random $k$-SAT formula on $n$ variables $x_1,\ldots,x_n$.
For a fixed number $B>1$ we let
$\mathcal{A}_B$ denote the property that a $k$-SAT formula $F$ on the variables $x_1,\ldots,x_n$
has less than $\frac12B^n$ satisfying assignments.
The following lemma shows that $\mathcal{A}_B$ has a sharp threshold.

\begin{lemma}\label{Lemma_sharpThr}
For any $B>1$ there is a sequence $(T_n^B)_{n\geq1}$ of integers such that for any $\epsilon>0$ we have
	\begin{eqnarray*}
	\lim_{n\rightarrow\infty}\Pr(F_k(n,(1-\epsilon)T_n^B)\mbox{ has property }\mathcal{A}_B)&=&0,\mbox{ and}\\
	\lim_{n\rightarrow\infty}\Pr(F_k(n,(1+\epsilon)T_n^B)\mbox{ has property }\mathcal{A}_B)&=&1.
	\end{eqnarray*}
\end{lemma}

\noindent
Assuming Lemma~\ref{Lemma_sharpThr}, we can infer Lemma~\ref{Thm_Amin} easily.

\noindent 
\begin{proof}[Lemma~\ref{Thm_Amin}.]
Let $r<2^k\ln2-k$.
Equations~(\ref{eq:b}) and~(\ref{epdef}) show that $\rho\mapsto\Lambda_b(1/2,k,\rho)$ is a continuous function.
Therefore, for any $\epsilon>0$ there is a $0<\delta<2^k\ln2-k-r$ such that $r'=(1+\delta)^2r$
satisfies
	$$%\Lambda_b(1/2,k,r')<
		\Lambda_b(1/2,k,r')>\Lambda_b(1/2,k,r)-\epsilon.$$
Let $B=\sqrt{\Lambda_b(1/2,k,r')}$.
Setting $t=\frac12B^n$, 
the second moment argument from~\cite{yuval} shows in combination
with the Paley-Zigmund inequality that
	$${\lim\inf}_{n\rightarrow\infty}\Pr [F_k(n,r'n)\mbox{ does not satisfy }\mathcal{A}_B]>0.$$
Therefore, Lemma~\ref{Lemma_sharpThr} implies that $r'n<(1+\delta)T_n^B$ for sufficiently large $n$.
Consequently, for large $n$ we have
	$rn=(1+\delta)^{-2}r'n<(1+\delta)^{-1}T_n^B.$
Hence, Lemma~\ref{Lemma_sharpThr} yields
	$$\lim_{n\rightarrow\infty}\Pr [F_k(n,rn)\mbox{ does not satisfy }\mathcal{A}_B]=1.$$
Thus, with probability $1-o(1)$ the number $Z$ of satisfying assignments of $F_k(n,rn)$ satisfies
	$$Z\geq\frac12B^n=\frac12\Lambda_b(1/2,k,r')^{n/2}\geq\frac12(\Lambda_b(1/2,k,r)-\epsilon)^{n/2}.$$
Since this is true for any $\epsilon>0$, the assertion follows.
\qed\end{proof}

\noindent 
\begin{proof}[Lemma~\ref{Lemma_numberkSAT}.]
As shown in~\cite{yuval}, the solution $\epsilon$
to~\eqref{epdef} satisfies
\begin{equation}\label{epbounds}
2^{1-k} + k4^{-k} < \epsilon < 2^{1-k} + 3k4^{-k} \enspace .
\end{equation}
Plugging these bounds into~(\ref{eq:b}) and performing a tedious but straightforward computation, we obtain that
	$$\frac12\ln\Lambda_b(1/2,k,r)\geq\ln2+r\brk{\ln(1-2^{-k})-k2^{3-2k}}.$$
Since $r\le 2^k$, the assertion thus follows from Lemma~\ref{Thm_Amin}.
\qed\end{proof}

To prove Lemma~\ref{Lemma_sharpThr}, we need a bit of notation.
If $\phi$ is a formula on a set of variables $y_1,\ldots,y_l$ disjoint from $x_1,\ldots,x_n$, then we let
$E_n(\phi)$ denote the set of all formulas that can be obtained from $\phi$ by substituting $l$
distinct variables among $x_1,\ldots,x_n$ for $y_1,\ldots,y_l$.
Moreover, for a formula $F$ on $x_1,\ldots,x_n$ we let
$F\oplus \phi=F\wedge\phi^*$, where $\phi^*$ is chosen uniformly at random from $E_n(\phi)$.
%In addition, for a sequence $m=(m_n)_{n\geq1}$ we say that $\phi$ is \emph{likely in $F_k(n,m)$} if
%with probability $\Omega(1)$ a random formula $F_k(n,m_n)$ contains an element of $E_n(M)$
%as a subformula.

Note that $\mathcal{A}_B$ is a \emph{monotone} property, i.e., if $F$ has the property $\mathcal{A}_B$ and
$F'$ is another formula on the variables $x_1,\ldots,x_n$, then $F\wedge F'$ has the property $\mathcal{A}_B$ as well.
Therefore, we can use the following theorem from Friedgut~\cite{EhudHunting} to prove by contradiction that $\mathcal{A}_B$ has a sharp threshold.
Let $\omega(n)=\lceil\log n\rceil$ for concreteness.

\begin{theorem}\label{Thm_Ehud}
Suppose that $\mathcal{A}_B$ does \emph{not} have a sharp threshold.
Then there exist a number $\alpha>0$, a formula $\phi$, and
for any $n_0>0$ numbers $N>n_0$, $M>0$ and a formula $F$ with variables $x_1,\ldots,x_N$
such that the following is true.
\begin{description}
\item[T1.] $\Pr(F\oplus\phi\mbox{ has the property }\mathcal{A}_B)>1-\alpha$.
\item[T2.] $\alpha<\Pr(F_k(N,M)\mbox{ has the property }\mathcal{A}_B)<1-3\alpha$.
\item[T3.] With probability at least $\alpha$ a random formula $F_k(N,M)$ contains
	an element of $E_N(\phi)$ as a subformula.
\item[T4.] $\Pr(F\wedge F_k(N,2\omega(N)))\mbox{ has the property }\mathcal{A}_B)<1-2\alpha$.
\end{description}
\end{theorem}

In the sequel we assume the existence of
$\alpha$, $\phi$, $N$, $M$, and $F$ satisfying conditions {\bf T1}--{\bf T3}.
To conclude that $\mathcal{A}_B$ has a sharp threshold, we shall show that then condition {\bf T4} cannot hold.
Clearly, we may assume that $N$ is sufficiently large (by choosing $n_0$ appropriately).
Let $V=\{x_1,\ldots,x_N\}$.

\begin{lemma}
The formula $\phi$ is satisfiable.
\end{lemma}
\begin{proof}
Any $k$-SAT formula %in which each variable occurs at most $k$ times is satisfiable.
%Therefore, each formula 
that contains at most as many clauses as variables is satisfiable.
Hence, to establish the lemma, we will show that
the probability $Q$ that $F_k(N,M)$ contains a subformula on $l$ variables
with at least $l$ clauses is smaller than $\alpha$;
then the assertion follows from {\bf T3}.

To prove this statement, we employ the union bound.
There are $\binom{N}{l}$ ways to choose a set of $l$ variables,
and $\binom{M}{l}$ ways to choose slots for the $l$ clauses of the subformula.
Furthermore, the probability that the random clauses in these $l$ slots
contain only the chosen variables is at most $(l/N)^{kl}$.
Hence, the probability that $F_k(N,M)$ has $l$ variables that span a subformula with
at least $l$ clauses is at most
	\begin{equation}\label{eqphiSAT}
	Q\leq{N\choose l}{M\choose l}(l/N)^{kl}\leq\left(\frac{e^2Ml^k}{N^2}\right)^l.%\leq\left(\frac{e^2M\log N}{N^2}\right)^l.
	\end{equation}
Further, {\bf T2} implies that $M/N\leq 2^k$, because 
for $M/N>2^k$ the expected number of satisfying assignments of $F_k(N,M)$ is less than $1$.
Thus, assuming that $N$ is sufficiently large,
we see that (\ref{eqphiSAT}) implies $Q\leq(e^2(2l)^k/N)^l<\alpha$, as claimed.
\qed\end{proof}

Thus, fix a satisfying assignment $\sigma:\{y_1,\ldots,y_l\}\rightarrow\{0,1\}$ of $\phi$.
Then we say that a satisfying assignment $\chi$ of $F$ is
\emph{compatible} with a tuple $(z_1,\ldots,z_l)\in V^l$ if
$\chi(z_i)=\sigma(y_i)$ for all $1\leq i\leq l$.
Furthermore, we call a tuple $(z_1,\ldots,z_l)\in V^l$ \emph{bad}
if $F$ has less than $\frac12B^N$ satisfying assignments $\chi$ that are compatible with $(z_1,\ldots,z_l)$.

\begin{lemma}\label{Lemma_manybad}
There are at least $(1-\alpha)N^l$ bad tuples.
\end{lemma}
\begin{proof}
The formula $F\oplus\phi$ is obtained by substituting $l$ randomly chosen variables
$(z_1,\ldots,z_l)\in V^l$ for the variables $y_1,\ldots,y_l$ of $\phi$ and adding the resulting clauses to $F$.
Since by {\bf T1}  with probability at least $1-\alpha$
the resulting formula has at most $\frac12B^N$ satisfying assignments,
a uniformly chosen tuple $(z_1,\ldots,z_l)\in V^l$ is bad with probability at least $1-\alpha$.
Thus, there are at least $(1-\alpha)N^l$ bad tuples.
\qed\end{proof}

\begin{lemma}\label{Lemma_ESbad}
With probability at least $1-\alpha$ a random formula $F_k(N,\omega(N))$
contains $l$ clauses $C_1,\ldots,C_l$ with the following two properties.
\begin{description}
\item[B1.] For each $1\leq i\leq l$ there is a $k$-tuple of variables $(v_i^1,\ldots,v_i^k)\in V^k$
	such that  $C_i=v_i^1\vee\cdots\vee v_i^k$ if $\sigma(i)=1$, and
	$C_i=\neg v_i^1\vee\cdots\vee\neg v_i^k$ if $\sigma(i)=0$.
\item[B2.] For any function $f:[l]\rightarrow[k]$ the $l$-tuple $(v_1^{f(1)},\ldots,v_l^{f(l)})$ is bad.
\end{description}
\end{lemma}

\noindent
The proof of Lemma~\ref{Lemma_ESbad} is based on the following version of the Erd\H{o}s-Simonovits theorem%~\cite{ErdosSimonovits} A
(cf.~\cite[Proposition~3.5]{EhudHunting}).

\begin{theorem}\label{Thm_ES}
For any $\gamma>0$ there are numbers $\gamma',\nu_0>0$ such that for any $\nu>\nu_0$ and any set $H\subset [\nu]^{l}$
of size $|H|\geq\gamma\nu^t$ the following is true.
If $l$ $k$-tuples $(w_1^1,\ldots,w_1^k),\ldots,(w_l^1,\ldots,w_l^k)\in [\nu]^k$ are chosen
uniformly at random and independently, then with probability at least $\gamma'$ for
any function $f: [l]\rightarrow [k]$ the tuple $(w_1^{f(1)},\ldots,w_l^{f(l)})$ belongs to $H$.
\end{theorem}

\begin{proof}[Proof of Lemma~\ref{Lemma_ESbad}]
Assuming that $N$ is sufficiently large, we apply Theorem~\ref{Thm_ES}
to $\gamma=1-\alpha$, $\nu=N$, and the set $H\subset [N]^l$ of bad $l$-tuples.
Then by Lemma~\ref{Lemma_manybad} we have $|H|\geq \gamma N^l$.
Now, consider $l$ random $k$-clauses $C_1,\ldots,C_l$ over the variable set $V$ chosen uniformly and independently.
Let $V_1,\ldots,V_l$ be the $k$-tuples of variables underlying $C_1,\ldots,C_l$.
Then Theorem~\ref{Thm_ES} entails that $V_1,\ldots,V_l$ satisfy condition {\bf B2} with probability at least $\gamma'$.
Moreover, given that this is the case, condition {\bf B1} is satisfied with probability $2^{-kl}$.
Therefore, the clauses $C_1,\ldots,C_l$ satisfy both {\bf B1} and {\bf B2} with probability at least $\gamma'2^{-kl}$.
Hence, the probability that $F_k(N,\omega(N))$ does not feature
an $l$-tuple of clauses satisfying {\bf B1} and {\bf B2} is at most $(1-\gamma'2^{-kl})^{\lfloor\omega(N)/l\rfloor}$.
Since $\omega(N)=\lceil\log N\rceil$, we can ensure that this expression is less than $\alpha$
by choosing $N$ large enough.
\qed\end{proof}

\begin{corollary}\label{Cor_ESbad1}
With probability at least $1-\alpha$ the formula $F\wedge F_k(N,\omega(N))$
has at most $\frac12k^l\cdot B^N$ satisfying assignments.
\end{corollary}
\begin{proof}
We will show that if $C_1,\ldots,C_l$ are clauses satisfying the two conditions
from Lemma~\ref{Lemma_ESbad}, then $F\wedge C_1\wedge\cdots\wedge C_l$ has at most $\frac12k^lB^N$ satisfying assignments.
Then the assertion follows from Lemma~\ref{Lemma_ESbad}.

Thus, let $\chi$ be a satisfying assignment of $F\wedge C_1\wedge\cdots\wedge C_l$.
Then by the {\bf B1} for each $1\leq i\leq l$ there is an index $f_{\chi}(i)\in [k]$ such that $\chi(v_i^{f_{\chi}(i)})=\sigma(i)$.
Moreover, by {\bf B2} the tuple $(v_1^{f_{\chi}(1)},\ldots,v_l^{f_{\chi}(l)})$ is bad.
Hence, the map $\chi\mapsto f_{\chi}\in [k]^l$ yields a bad tuple $(v_i^{f_{\chi}(i)})_{1\leq i\leq l}$ for each satisfying assignment.
Therefore, %by the definition of ``bad''
the number of satisfying assignments mapped to any tuple in $[k]^l$ is at most $\frac12B^n$.
Consequently, $F\wedge C_1\wedge\cdots\wedge C_l$ has at most $\frac12k^l\cdot B^n$ satisfying assignments in total.
\qed\end{proof}

\begin{corollary}\label{Cor_ESbad2}
With probability at least $1-\frac32\alpha$ the formula $F\wedge F_k(N,2\omega)$
satisfies $\mathcal{A}_B$.
\end{corollary}
\begin{proof}
The formula $F^{**}=F\wedge F_k(N,2\omega)$ is obtained from $F$ by attaching $2\omega(N)$ random clauses.
Let $F^*=F\wedge F_k(N,\omega(N))$ be the formula resulting by attaching the first $\omega(N)$ random clauses.
Then by Corollary~\ref{Cor_ESbad1} with probability at least $1-\alpha$
the formula $F^*$ has at most $\frac12k^l\cdot B^N$ satisfyng assignments.
Conditioning on this event, we form $F^{**}$ by attaching another $\omega(N)$ random clauses to $F^*$.
Since for any satisfying assignment of $F^*$ the probability that these additional $\omega(N)$
clauses are satisfied as well is $(1-2^{-k})^{\omega(N)}$, the expected number of satisfying assignments of $F^{**}$ is at most
	$$\frac12k^l\cdot B^N\cdot(1-2^{-k})^{\omega(N)}\leq\frac{\alpha}4\cdot B^N,$$
provided that $N$ is sufficiently large.
Therefore, Markov's inequality entails that
	$$\Pr(F^{**}\mbox{ violates }\mathcal{A}_B|F^*\mbox{ has at most $\frac12k^l\cdot B^N$ satisfying assignments})\leq\alpha/2.$$
Thus, we obtain
	\begin{eqnarray*}
	\Pr(F^{**}\mbox{ violates }\mathcal{A}_B)
		&\leq&\Pr(F^*\mbox{ has more than $\frac12k^l\cdot B^N$ satisfying assignments})\\
		&&\quad+\Pr(F^{**}\mbox{ violates }\mathcal{A}_B
			|F^*\mbox{ has at most $\frac12k^l\cdot B^N$ satisfying assignments})
		\leq3\alpha/2,
	\end{eqnarray*}
as desired.
\qed\end{proof}

\noindent
Combining Theorem~\ref{Thm_Ehud} and Corollary~\ref{Cor_ESbad2}, we conclude that
$\mathcal{A}_B$ has a sharp threshold, thereby completing the proof of Lemma~\ref{Lemma_sharpThr}.

\subsection{Proof of Theorem~\ref{Thm_icy_sat}}

Using Theorem~\ref{Thm_SATExchange}, we shall establish the following lemma.

\begin{lemma}\label{Lemma_icy_sat}
There exist numbers $0<\alpha_1<\alpha_2<\frac13$, $\lambda>0$, and $\gamma>0$ such that
a random pair $(F,\sigma)$ chosen from the distribution $\mathcal{U}_{n,m}$ has the following two
properties \whp
\begin{enumerate}
\item Any assignment $\tau$ such that $\alpha_1<n^{-1}\dist(\sigma,\tau)<\alpha_2$
	satisfies $H(\tau)\geq\lambda n$.
\item $|\{\tau\in \SSS(F):\dist(\sigma,\tau)<\beta_2n\}|<2^n(1-2^{-k})^{m}\exp(-\gamma n-k2^{3-k}n)$.
\end{enumerate}
\end{lemma}

\begin{proof}[Theorem~\ref{Thm_icy_sat}]
Let $F=F_k(n,m)$ be a random $k$-SAT instance.
To each assignment $\sigma\in \SSS(F)$ we assign the set
	$$\mathcal{C}_{\sigma}=\{\tau\in \SSS(F):\dist(\sigma,\tau)<\alpha_2 n\}.$$
Due to Lemma~\ref{Lemma_numberkSAT}, 
a similar argument as in
the proof of Theorem~\ref{Thm_icy_col} in Section~\ref{Sec_icy_col} yields Theorem~\ref{Thm_icy_sat}.
\qed\end{proof}

Let $\lambda>0$ be small but fixed.
Let $F=F_k(n,m)$ be a random $k$-SAT formula with $m=rn$ clauses.
Then for any $\sigma\in\{0,1\}^n$ we have
	$$n^{-1}\ln\pr\brk{\sigma\mbox{ is satisfying}}=r\ln(1-2^{-k}),$$
because of the independence of all $m$ clauses.
%In fact, there is a function $\psi(\lambda)$ such that $\lim_{\lambda\rightarrow0}\psi(\lambda)=0$ and
%	$$n^{-1}\ln\pr\brk{H(\sigma)<\lambda n}=r\ln(1-2^{-k})+\psi(\lambda).$$
Furthermore, if $\tau\in\{0,1\}^n$ is a second assignment at Hamming distance $\alpha n$ from $\sigma$,
then
	$$n^{-1}\ln\pr\brk{\sigma,\tau\mbox{ are both satisfying}}=r\ln(1-2^{1-k}+2^{-k}(1-\alpha)^k).$$
Indeed, there is a function $\Psi(\lambda)$ 
such that $\lim_{\lambda\rightarrow0}\Psi(\lambda)=0$ and
	$$n^{-1}\ln\pr\brk{H(\sigma)=0\wedge H(\tau)\leq\lambda n}=r\brk{\ln(1-2^{1-k}+2^{-k}(1-\alpha)^k)+\Psi(\lambda)}.$$
Let $X_{\alpha,\lambda}$ signify the number of assignments $\tau$ at Hamming distance
$\alpha n$  from $\sigma$ such that $H(\tau)\leq\lambda n$.

\begin{lemma}\label{Lemma_Xalpha}
There is a number $0<\alpha^*<1/3$ such that
for sufficiently small $\lambda>0$ we have
	$$n^{-1}\ln \Erw\brk{X_{\alpha^*,\lambda}|\sigma\mbox{ is satisfying}}
		<-k2^{3-k}.$$
\end{lemma}
\begin{proof}
There are $\bink{n}{\alpha n}$ ways to choose an assignment $\tau$
at Hamming distance $\alpha n$ from $\sigma$.
Therefore, due to the formulas derived above, we have
	$$n^{-1}\ln \Erw\brk{X_{\alpha,\lambda}|\sigma\mbox{ is satisfying}}
		=-\alpha\ln\alpha-(1-\alpha)\ln(1-\alpha)+r\brk{\ln\bc{1-\frac{1-(1-\alpha)^k}{2^k-1}}+\Psi(\lambda)+o(1)}.$$
Setting $\alpha^*=(k\ln k)^{-1}$ and simplifying, we obtain the assertion.
\qed\end{proof}

\begin{corollary}\label{Cor_Xalpha}
There are numbers $\lambda>0$ and $0<\alpha_1<\alpha_2<\frac13$ such that
with probability at least $1-o(\exp(-k2^{3-k}n)$ in a pair $(F,\sigma)\in\Lambda_{n,m}$ chosen from the
distribution $\mathcal{P}_{n,m}$ there is no
assignment $\tau$ such that
such that $H(\tau)<\lambda n$ and $\alpha_1<n^{-1}\dist(\sigma,\tau)<\alpha_2$.
\end{corollary}
\begin{proof}
If $(F,\sigma)$ is chosen from $\mathcal{P}_{n,m}$,
then $F$ is distributed as a random formula $F_k(n,m)$ given that $\sigma$ is a satisfying assignment.
Therefore, the corollary follows from Lemma~\ref{Lemma_Xalpha} and Markov's inequality,
where we use the fact that $\alpha\mapsto n^{-1}\ln\Erw\brk{X_{\alpha,\lambda}|\sigma\mbox{ is satisfying}}$ is a continuos
function.
\qed\end{proof}

Furthermore, the following estimate has been established in~\cite{fede}.

\begin{lemma}\label{Lemma_XalphaUpper}
We have
	$\max_{0\leq \alpha\leq\frac13}
		n^{-1}\ln \Erw\brk{X_{\alpha,\lambda}|\sigma\mbox{ is satisfying}}
		<\ln2+r(1-2^{-k})-2\exp(k2^{3-k}).$
\end{lemma}

Finally, Lemma~\ref{Lemma_icy_sat} follows from Theorem~\ref{Thm_SATExchange} in combination
with Corollary~\ref{Cor_Xalpha} and Lemma~\ref{Lemma_XalphaUpper}.

\subsection{Proof of Theorem~\ref{thm:froz} ($k$-SAT)}

If $F$ is a $k$-SAT formula and $\sigma$ an assignment,
then we say that a variable $x$ \emph{supports} a clause $C$
if changing the value of $x$ would render $C$ unsatisfied.
Suppose that $k$ is sufficiently large and $(1+\eps)2^kk^{-1}\ln k<r=m/n<(1-\eps)2^k\ln 2$.
Let $\gamma,\delta>0$ be sufficiently small numbers.

\begin{lemma}\label{Lemma_kSATfroz}
A pair $(F,\sigma)$ chosen from $\mathcal{U}_{n,m}$
has the following property \whp\
	\begin{equation}\label{eqU}
	\parbox{15cm}{
	There is a set $U$ of at least $(1-\delta)n$ variables
	such that each variable in $U$ supports  $\gamma\ln k$ clauses
	$e$ that contain no variable from $V\setminus U$.}
	\end{equation}
\end{lemma}

\begin{proof}[Theorem~\ref{thm:froz}]
Let $\zeta>0$ signify a sufficiently small constant.
Let $(F,\sigma)$ be chosen from the distribution $\mathcal{U}_{n,m}$.
We may assume that the random pair $(F,\sigma)$ satisfies~(\ref{eqU}).
Moreover, a 1st moment computation shows that the random formula $F$ has the following property \whp
\begin{equation}\label{eqkSATdisc}
	\parbox{15cm}{
	There is no set $Z$ of variables of size $|Z|\leq\zeta n$ such that
	$F$ features at least $|Z|\gamma\ln k$ clauses $e$ that contain at least
	two variables from $Z$.}
	\end{equation}
Now, assume for contradiction that there is a satisfying assignment $\tau$ of $F$
such that the set $Z=\{v\in U:\tau(v)\not=\sigma(v)\}$ has size
$1\leq|Z|\leq\zeta n$.
Each $v\in Z$ supports in $\sigma$ at least $\gamma\ln k$ clauses $e$ that contain no variable from $V\setminus U$.
Since these clauses $e$ are satisfied in $\tau$, although $\tau(z)=0$,
each such $e$ contains another variable $w\not=v$ from $Z$.
Hence, $F$ contains at least $|Z|\gamma\ln k$ clauses $e$ containing at
least two variables from $Z$, in contradiction to~(\ref{eqkSATdisc}).
\qed\end{proof}

Lemma~\ref{Lemma_kSATfroz} is an immediate consequence of Theorem~\ref{Thm_SATExchange}
and the following lemma.

\begin{lemma}\label{Lemma_kSATfrozpl}
A pair $(F,\sigma)$ chosen from $\mathcal{P}_{n,m}$
has the property~(\ref{eqU}) with probability $1-o(\exp(-k2^{3-k}n))$.
\end{lemma}
\begin{proof}
We may assume that $r=m/n=(1+\eps)2^kk^{-1}\ln k$ for a fixed $\eps>0$.
Moreover, without loss of generality, we may assume that
the assignment $\sigma$ sets all variables $V=\{x_1,\ldots,x_n\}$ to true.
Let $F$ denote a random formlua with $m$ clauses satisfied by $\sigma$,
and let $\Xi$ signify the set of all uniquely satisfied clauses of $F$.
Consider the following process.
\begin{enumerate}
\item Let $Z_0$ be the set of all variables $x$ that support fewer than $2\gamma\ln k$ clauses.
\item Let $Z=Z_0$.
	While there is a variable $x\in V\setminus Z$ that
	supports at least $\gamma\ln k$ clauses from $\Xi$ that contain a variable from $Z$,
	add $x$ to $Z$.
\end{enumerate}

The \emph{expected} number of uniquely satisfied clauses is at least $k2^{-k}m\geq(1+\eps)n\ln k$.
Hence, each variable is expected to support at least $(1+\eps)\ln k$ clauses.
Therefore, if $\gamma>0$ is sufficiently small, then there is a contant $\beta>0$ such that
the probability that a variable $x$ supports fewer than $2\gamma\ln k$ clauses
is at most $k^{-1-\beta}$.
Hence, by Chernoff bounds we have $|Z_0|\leq n k^{-1-\beta/2}$
with probability at least $1-o(\exp(-k2^{3-k}n))$.

Thus, assume that $|Z_0|\leq n k^{-1-\beta/2}$.
We claim that then the final set $Z$ resulting from Step~2 has size at most  $|Z|\leq 2n k^{-1-\beta/2}$.
For assume that $|Z|>2n k^{-1-\beta/2}$.
Then Step~2 removed at least $|Z|/2$ variables,
whence there are at least $\gamma\ln k|Z|/2$ clauses $e\in\Xi$ that contain two variables from $Z$.
However, a standard 1st moment argument shows that the probability that there exists a set
$Z$ with this property is $o(\exp(-k2^{-k}n))$.
Hence, with probability at least $1-o(\exp(-k2^{-k}n)$ we have $|Z|\leq 2n k^{-1-\beta/2}$.
Setting $U=V\setminus Z$ concludes the proof.
\qed\end{proof}

\end{appendix}

\end{document}